\documentclass[]{amsart}

\usepackage[utf8]{inputenc}
\usepackage[margin=1in]{geometry}

\usepackage{amsthm, amsfonts, amsmath, makecell, tikz, %pgfplots,
 graphicx, colortbl}
\usepackage{multirow, booktabs, bigstrut}

\usetikzlibrary{graphs,graphs.standard,arrows.meta,calc,intersections}

\usepackage[]{algorithm, algorithmic}

\usepackage{tabularx}

\theoremstyle{plain}
\newtheorem{theorem}{Theorem}[section]
\newtheorem{lemma}[theorem]{Lemma}

\newtheorem{proposition}[theorem]{Proposition}

\newtheorem{remark}[theorem]{Remark}

\newtheorem*{claim*}{Claim}

\theoremstyle{definition}
\newtheorem{definition}[theorem]{Definition}
\newtheorem{example}[theorem]{Example}

\DeclareMathOperator{\Aut}{Aut}
\DeclareMathOperator{\PGL}{PGL}
\DeclareMathOperator{\GL}{GL}

\DeclareMathOperator{\Doubles}{Doubles}
\DeclareMathOperator{\OL}{OL}
\DeclareMathOperator{\OLExt}{OLExt}
\DeclareMathOperator{\OLExtArrs}{OLExtArrs}

\newcommand{\R}{{\mathbb R}}
\newcommand{\C}{{\mathbb C}}
\newcommand{\Q}{{\mathbb Q}}

\newcommand{\Z}{{\mathbb Z}}

\definecolor{gray}{rgb}{.5,.5,.5}

\definecolor{black}{rgb}{0,0,0}

\definecolor{blue}{rgb}{0,0,1}

\definecolor{red}{rgb}{1,0,0}

\definecolor{green}{rgb}{0,1,0}

\definecolor{gold}{rgb}{.5,.5,.2}

\definecolor{yellow}{rgb}{1,1,.4}

\definecolor{purple}{rgb}{.5,0,.5}

\definecolor{darkgreen}{rgb}{0,.5,0}

\definecolor{orange}{rgb}{1,.55,0}

\definecolor{white}{rgb}{1,1,1}

\usepackage{todonotes}
\usepackage{hyperref}

\usepackage{subfigure}
\title{Moduli Spaces of One-Line Extensions of $(10_3)$ Configurations}
\author{Moshe Cohen}
\address{Mathematics Department, State University of New York at New Paltz, New Paltz, New York, USA}
\email{cohenm@newpaltz.edu}

\author{Baian Liu}
\address{Department of Mathematics, The Ohio State University, Columbus, Ohio}
\email{liu.7358@osu.edu}

\begin{document}

	\begin{abstract}
Two line arrangements in $\mathbb{CP}^2$ can have different topological properties even if they are combinatorially isomorphic.  Results by Dan Cohen and Suciu and by Randell show that a reducible moduli space under complex conjugation is a necessary condition.

We present a method to produce many examples of combinatorial line arrangements with a reducible moduli space obtained from a set of examples with irreducible moduli spaces. 

		In this paper, we determine the reducibility of the moduli spaces of a family of arrangements of 11 lines constructed by adding a line to one of the ten $(10_3)$ configurations. Out of the four hundred ninety-five combinatorial line arrangements in this family, ninety-five have a reducible moduli space, seventy-six of which are still reducible after the quotient by complex conjugation.
	\end{abstract}

%14N20  	Configurations and arrangements of linear subspaces
%14H37 		Automorphisms of curves
%14D06  	Fibrations, degenerations in algebraic geometry -- probably not
%14Q05  	Computational aspects of algebraic curves
%32S05  	Local complex singularities 
%32S22 Relations with arrangements of hyperplanes [See also 52C35]
%52C35  	Arrangements of points, flats, hyperplanes (aspects of discrete geometry)

	\maketitle

\section{Introduction}
Hyperplane arrangements are not only classically interesting \cite{os} but have been studied more recently \cite{dimca}.
  These objects have interesting combinatorial structures
\cite{oriented, levigraph} and even have applications to database searches \cite{database}.  Interesting questions concerning hyperplane arrangements include those regarding whether the combinatorics determine topological properties \cite{dimca:free}. 

We focus on hyperplane arrangements in two dimensions: line arrangements. We define a \textbf{combinatorial line arrangement} $\mathcal{A} = (\mathcal{P}, \mathcal{L})$ to be a finite set of \textbf{points} $\mathcal{P}$ and a finite set of \textbf{lines} $\mathcal{L}$, which are subsets of $\mathcal{P}$. We also require that the intersection of two distinct lines be one point or empty.  Given a combinatorial line arrangement $\mathcal{A}$, we consider its \textbf{moduli space} $\mathcal{M}_\mathcal{A}$, its set of all geometric realizations in $\mathbb{CP}^2$.  Elements of the same moduli space are combinatorially equivalent. 

The number of components of a moduli space gives us information about the topology of the arrangement.  In an irreducible moduli space or in a single connected component, Randell's Isotopy Theorem \cite{randell} states that any two arrangements $\mathcal{A}_1$ and $\mathcal{A}_2$ in 1-parameter family of combinatorially equivalent arrangements have diffeomorphic complements and that  ($\mathbb{CP}^2$,$\mathcal{A}_1$) is homeomorphic to ($\mathbb{CP}^2$,$\mathcal{A}_2$).  Furthermore, a result by Dan Cohen and Suciu \cite[Theorem 3.9]{CS} states that complex conjugate arrangements have equivalent braid monodromies and so also have diffeomorphic complements.

Thus we seek combinatorial line arrangements with a reducible moduli space $\mathcal{M}_\mathcal{A}$ and a reducible moduli space modulo complex conjugation $\mathcal{M}_\mathcal{A}^\C$.  
We construct numerous such examples of combinatorial line arrangements using one-line extensions. We say that a \textbf{one-line extension} of a combinatorial line arrangement $\mathcal{A}$ is $\mathcal{A}$ together with an additional line.  More commonly in the literature, matroids are extended by single elements;  according to Oxley, this ``can be fraught with difficulty'' \cite{oxley}.  Our method is a specific type of single-element extension, in a dual sense.  For more on matroids, see Oxley's textbook \cite{matroid}.  Kocay \cite{one} uses a method of one-point extensions in order to find coordinatizations of arrangements.

We apply one-line extensions to $(10_3)$ configurations. More generally, an \textbf{($\bf n_3$) configuration} is a combinatorial line arrangement with $n$ lines and $n$ triple points, which are points that lie in exactly three lines or, in other words, have multiplicity three.  Such arrangements are not only classically interesting \cite{martinetti, daub, gropp, sw} but have been studied more recently \cite{grunbaum, revisited, ProjPlane, one, 133}.

The reason we apply one-line extension to $(10_3)$ configurations is to find combinatorial line arrangements with 11 lines and reducible moduli space. Combinatorial line arrangements with 10 or fewer lines and a reducible moduli space have been classified. In 1997, Fan showed that there are no combinatorial line arrangements with a reducible moduli space of up to and including 6 lines \cite{fan}. Garber, Teicher, and Vishne showed that there are combinatorial line arrangements with a reducible moduli space of up to 8 lines that are realizable in $\mathbb{RP}^2$ \cite{gtv}. Nazir and Yoshinaga verified that for combinatorial line arrangements of up to and including 9 lines, those with a reducible moduli space must contain one of three combinatorial line arrangements as a subarrangement: the MacLane arrangement, also known as the M\"obius-Kantor arrangement or the unique ($8_3$) configuration  \cite{Kantor, Mobius1828, Reye1882, Schroter1889, Mac}; the Nazir-Yoshinaga arrangement \cite{ny}; and the Falk-Sturmfels arrangement \cite[cited as unpublished]{CS}.  
We say an arrangement is \textbf{exceptional} if it contains one of these three arrangements as a subarrangement and \textbf{unexceptional} otherwise.

Suppose a combinatorial line arrangement $\mathcal{A}$ contains a line $\ell$ that passes through at most two points of multiplicity three or greater. Form another combinatorial line arrangement $\mathcal{A}\backslash\ell$ by deleting $\ell$ from $\mathcal{A}$.  Nazir and Yoshinaga show that if the moduli space of $\mathcal{A}\backslash\ell$ is irreducible, then the moduli space of $\mathcal{A}$ is irreducible \cite[Lemma 3.2]{ny}. 
  We say a combinatorial line arrangement is \textbf{reductive} if it contains a line that passes through at most two points of multiplicity three or greater and \textbf{non-reductive} otherwise.

The first author with Amram, Teicher, and Ye completed the classification of irreducibility of the moduli space of non-reductive, unexceptional arrangements of 10 lines in \cite{aty} and \cite{amram}, producing eighteen examples.  
Further analysis of these authors together with Sun and Zarkh reduced the number of candidates from eighteen to fifteen \cite{acstyz}.  
Motivating this work, two of the nine combinatorial line arrangements from \cite{amram} with a reducible moduli space are on the list of the eleven one-line extensions of ($9_3$) configurations.

Amram, Gong, Teicher, and Xu classify the moduli spaces of non-reductive arrangements of 11 lines with at least one point of multiplicity at least five in \cite{agtx}, identifying thirty-eight arrangements that satisfy the necessary moduli space condition.  
 Further interesting examples of arrangements of 11 lines from the literature include a reductive, exceptional example by Artal Bartolo, Carmona Ruber, Cogolludo-Agust\'in, and Marco Buzun\'ariz and twenty-nine examples by Guerville-Ball\'e \cite{GB29}.

	In this current work, we describe the one-line extension construction. While it does not produce reducible moduli spaces all of the time, it can empirically produce reducible moduli spaces a lot of the time. We use this to continue the classification of moduli spaces of non-reductive combinatorial line arrangements of 11 lines.

\medskip

\textbf{Main Results. }  Given a $(10_3)$ configuration, we determine the number of possible ways to add an eleventh line through some number of the double points:  at least three so that the arrangement is not reductive; and at most five because a line through six existing doubles must belong to an arrangement of at least 13 lines.  Over the ten $(10_3)$ configurations, this gives: a subtotal of three hundred thirty-six arrangements, fifteen of which appear twice, for three double points; a subtotal of one hundred eighty-eight arrangements, thirty-seven of which appear twice, 
 for four double points; and a total of twenty-three arrangements for five double points.  We then classify the moduli spaces of these arrangements.  A summary of these results can be found in Tables \ref{tab:main3}, \ref{tab:main4}, and \ref{tab:main5}.

\begin{theorem}
\label{thm:main}
Out of the three hundred twenty-one distinct arrangements obtained by adding an eleventh line through three double points in one of the ten $(10_3)$ configurations, just one of them has a reducible moduli space modulo complex conjugation:  $(10_3)_7.ADO$ as discussed in Example \ref{ex:ADO}.
\end{theorem}

\begin{theorem}
\label{thm:four}
Out of the one hundred fifty-one 
 distinct arrangements obtained by adding an eleventh line through four 
 double points 
 in one of the ten $(10_3)$ configurations, seventy-four of them have a reducible moduli space modulo complex conjugation.  
These are listed in Table \ref{tab:74cases} at the end of the Introduction.
\end{theorem}

\begin{theorem}
\label{thm:five}
Out of the twenty-three distinct arrangements obtained by adding an eleventh line through five double points in one of the ten $(10_3)$ configurations, just one of them has a reducible moduli space modulo complex conjugation:  $(10_3)_1.AEIKO$ as discussed in Example \ref{ex:AEIKO}.
\end{theorem}

{\renewcommand{\arraystretch}{1.2}\begin{center}
\begin{table}[!htb]
	\scalebox{1}{
		\begin{tabular}{|l|c|c|c|c|c|c|c|c|c|c||c|c|}
			\hline
%			Configuration $(10_3)_j$ & 1 & 2 & 3 & 4 & 5 & 6 & 7 & 8 & 9 & 10 & Subtotal & Total \\ \hline
			\hspace{5mm} $j$ & 1 & 2 & 3 & 4 & 5 & 6 & 7 & 8 & 9 & 10 & Subtotal & Total \\ \hline
%			Arrangements $\mathcal{A}$ & 4 & 17 & 42 & 11 & 76 & 30 & 50 & 50 & 39 & 17 & 336 & 321 \\ \hline
			$\#$ arrangements constructed from $(10_3)_j$ & 4 & 17 & 42 & 11 & 76 & 30 & 50 & 50 & 39 & 17 & 336 & 321 \\ \hline \hline
%			Connected $\mathcal{M}_\mathcal{A} \neq \emptyset$ & 3 & 13 & 32 & 0 &73 & 26 & 41 & 48 & 35 & 17 & 288 & 280 \\ \hline
			$\#$ with irreducible, non-empty moduli space & 3 & 12 & 34 & 0 &73 & 25 & 43 & 48 & 36 & 17 & 291 & 282 \\ \hline
%			$\mathcal{M}_\mathcal{A} = \emptyset$ & 1 & 4 & 9 & 11 &3 &2 & 8 & 2 &4 &0 &44 & 37 \\ \hline
			$\#$ with empty moduli space & 1 & 5 & 7 & 11 &3 &4 & 6 & 2 &3 &0 &42 & 36 \\ \hline
%			Disconnected $\mathcal{M}_\mathcal{A}$, connected $\mathcal{M}_\mathcal{A}^\mathbb{\C}$ &0&0&1&0&0&2&0&0&0&0&3& 3\\ \hline
			$\#$ with reducible $\mathcal{M}_\mathcal{A}$ but irreducible $\mathcal{M}_\mathcal{A}^\C$ &0&0&1&0&0&1&0&0&0&0&2&2\\ \hline
%			Disconnected $\mathcal{M}_\mathcal{A}^\mathbb{\C}$ & 0 & 0 &  0 & 0 & 0 &0 & \cellcolor{yellow}  1 &0 & 0 &0 & 1 & 1 \\ 
			$\#$ with reducible $\mathcal{M}_\mathcal{A}^\C$ & 0 & 0 &  0 & 0 & 0 &0 & \textbf{1} &0 & 0 &0 & \textbf{1} & \textbf{1} \\ 
\hline
		\end{tabular}
	}
	\label{tab:main3}
	\caption{Classification of the moduli space of arrangements obtained by adding a line through \emph{three} double points of a $(10_3)$ configuration}\end{table}
\end{center}
}

{\renewcommand{\arraystretch}{1.2}
\begin{table}[!htbp]
	\begin{center}
		\scalebox{1}{
			\begin{tabular}{|l|c|c|c|c|c|c|c|c|c|c||c|c|}
				\hline
				\hspace{5mm} $j$ 
				& 1 & 2 & 3 & 4 & 5 & 6 & 7 & 8 & 9 & 10 & Subtotal & Total \\ \hline
				$\#$ arrangements constructed from $(10_3)_j$ 
				& 2 & 8 & 21 & 5 & 45 & 16 & 25 & 30 & 24 & 12 & 188 & 151 \\ \hline \hline
			$\#$ with irreducible, non-empty moduli space 					
				& 1 & 2 & 3 & 0 & 0 & 2 & 3 & 1 & 0 & 1 & 13 & 10 \\ \hline
			$\#$ with empty moduli space 												
				& 0 & 4 & 8 & 5 & 13 & 8 & 8 & 9 & 11 & 1 & 67 & 50 \\ \hline
			$\#$ with reducible $\mathcal{M}_\mathcal{A}$ but irreducible $\mathcal{M}_\mathcal{A}^\C$						
				& 0 & 1 & 4 & 0 & 5 & 2 & 1 & 4 & 2 & 3 & 22 & 17 \\ \hline
			$\#$  with reducible $\mathcal{M}_\mathcal{A}^\C$ 									& \textbf{1} & \textbf{1} & \textbf{6} & 0 & \textbf{27} & \textbf{4} & \textbf{13} & \textbf{16} & \textbf{11} & \textbf{7} & \textbf{86} & \textbf{74} \\ \hline
			\end{tabular}
		}
	\end{center}
	\caption{Classification of the moduli space of arrangements obtained by adding a line through \emph{four} double points of a $(10_3)$ configuration}
\label{tab:main4}
\end{table}
}

{\renewcommand{\arraystretch}{1.2}
\begin{table}[!htbp]
	\begin{center}
		\scalebox{1}{
			\begin{tabular}{|l|c|c|c|c|c|c|c|c|c|c||c|}
				\hline
				\hspace{5mm} $j$ & 1 & 2 & 3 & 4 & 5 & 6 & 7 & 8 & 9 & 10 & Total \\ \hline
				$\#$ arrangements constructed from $(10_3)_j$& 1 & 1 & 2 & 1 & 5 & 2 & 2 & 3 & 3 & 3 & 23 \\ \hline \hline
				$\#$ with irreducible, non-empty moduli space & 0 & 0 & 0 & 0 & 0 & 0 & 0 & 0 & 0 & 1 & 1\\ \hline
%				Connected $\mathcal{M}_\mathcal{A}$ & 0 & 0 & 0 & 0 & 0 & 0 & 0 & 0 & 0 & 1 & 1\\ \hline
				$\#$ with empty moduli space & 0 & 1 & 2 & 1 & 5 &2 &2 &3 &3 &2 & 21\\ \hline
%				$\mathcal{M}_\mathcal{A} = \emptyset$ & 0 & 1 & 2 & 1 & 5 &2 &2 &3 &3 &2 & 21\\ \hline
				$\#$ with reducible $\mathcal{M}_\mathcal{A}$ but irreducible $\mathcal{M}_\mathcal{A}^\C$	 & 0 & 0 & 0 & 0 & 0 & 0 & 0 & 0 & 0 & 0 & 0 \\ \hline
% 			$\mathcal{M}_\mathcal{A}$ not connected but $\mathcal{M}_\mathcal{A}^\mathbb{C}$ connected& 0 & 0 & 0 & 0 & 0 & 0 & 0 & 0 & 0 & 0 & 0 \\ \hline
				$\#$ with reducible $\mathcal{M}_\mathcal{A}^\C$ 	 & \textbf{1} & 0 & 0 & 0 & 0 & 0 & 0 & 0 & 0 & 0 & \textbf{1}\\ \hline
%				Disconnected $\mathcal{M}_\mathcal{A}^\mathbb{C}$ & 1 & 0 & 0 & 0 & 0 & 0 & 0 & 0 & 0 & 0 & 1\\ \hline
			\end{tabular}
		}
	\end{center}
	\caption{Classification of the moduli space of arrangements obtained by adding a line through \emph{five} double points of a $(10_3)$ configuration}
	\label{tab:main5}
\end{table}
}

\begin{remark}
	All of the $(10_3)$ configurations have a 2-dimensional moduli space, except for $(10_3)_1$ and $(10_3)_4$. The configuration $(10_3)_1$ has a 3-dimensional moduli space, and $(10_3)_4$ has an empty moduli space. Introducing a new line through 4 double points gives 2 restraints on the moduli space. If neither of these restraints are already present, then the moduli space is 0-dimensional or empty, and a 0-dimensional moduli space is reducible as long as it contains more than one point. This is empirically why one-line extensions of $(10_3)$ configurations through 4 double points yield a high proportion of reducible moduli spaces.
\end{remark}

\begin{remark}
Our example $(10_3)_6.AFIO$ appears as $C_{28}$ in a recent work by Guerville-Ball\'e \cite{GB29}; it is the only such overlap.
\end{remark}

\begin{example}
\label{ex:ADO}
Consider the $(10_3)$ configuration $(10_3)_7$ with an eleventh line passing through the intersections $L_1\cap L_5$, $L_2\cap L_4$, and $L_9\cap L_{10}$.  We call this arrangement $(10_3)_7.ADO$. Its arrangement table is given in Table \ref{tab:ADO}, and two of its geometric realizations from the two different irreducible components of its moduli space are given in Figure \ref{fig:ADO}.

%\begin{center}
	\begin{table}[!htp]
%		\centering
		\begin{tabular}{ccccccccccc}
			$L_1$ & $L_2$ & $L_3$ & $L_4$ & $L_5$ & $L_6$ & $L_7$ & $L_8$ & $L_9$ & $L_{10}$ & $L_{11}$\\ \hline
			1&1&1&2&4&6&5&3&7&2&$A$\\
			2&4&6&8&8&9&7&5&3&4&$D$\\
			3&5&7&9&0&0&8&9&0&6&$O$\\
			$A$&$D$& &$D$&$A$& & & &$O$&$O$&
		\end{tabular}
	\caption{The arrangement table for $(10_3)_7.ADO$, which has a Galois conjugate moduli space}%disconnected $\mathcal{M}_\mathcal{A}^\C$.}
	\label{tab:ADO}
	\end{table}
%\end{center}

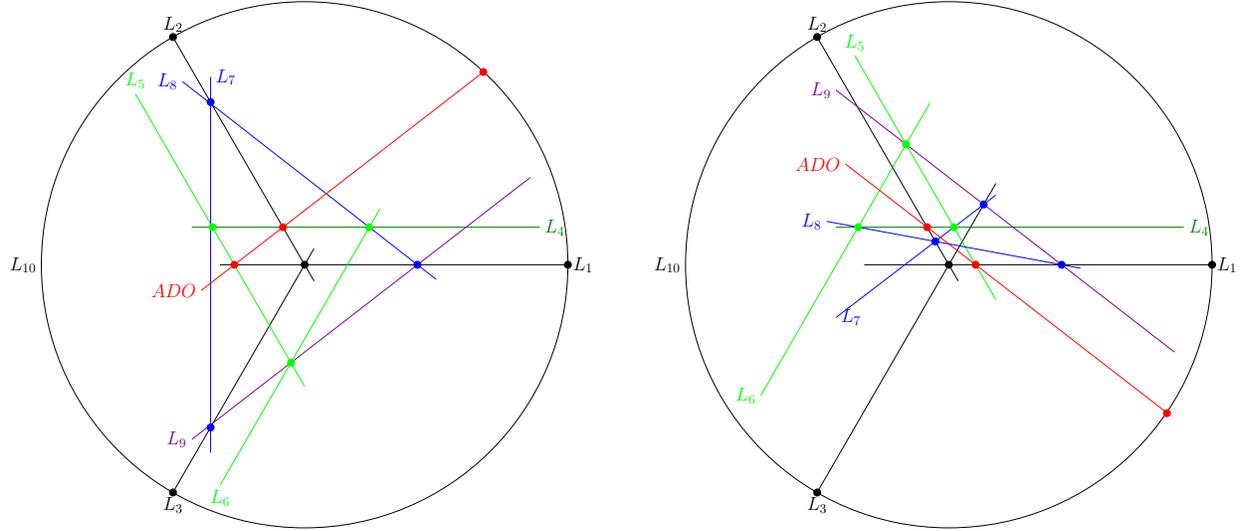
\begin{figure}[!htb]
\begin{center}
	{\large
		\scalebox{.5}{
			\begin{tikzpicture}[domain=-5:5,scale=1]

%			\draw (0,0) circle (5cm);
			\draw (0,0) circle (7cm);
			\draw (-7,0) node[left] {\LARGE $L_{10}$};
			\draw[domain=-2.25:7] plot (\x, 0) node[right]{\LARGE $L_1$};
			\draw[domain=.25:-3.5] plot (\x, -1.73* \x) node[above]{\LARGE $L_2$};
			\draw[domain=.25:-3.5] plot (\x,  1.73*\x) node[below]{\LARGE $L_3$};

			\fill[color=black] (0,0) circle (3pt);
			\fill[color=black] (7,0) circle (3pt);
			\fill[color=black] (-3.5,6.06) circle (3pt);
			\fill[color=black] (-3.5,-6.06) circle (3pt);

			\draw[color=darkgreen, domain=-3:6.25] plot (\x, 1) node[right]{\LARGE $L_4$};
			\draw[color=green, domain=0:-4.5] plot (\x, -1.73* \x-3.23607) node[above]{\LARGE $L_5$};
			\draw[color=green, domain=2:-2.25] plot (\x,  1.73*\x-1.97533) node[below]{\LARGE $L_6$};

%			\draw[color=blue, domain=-3:5] plot (\x, -289*\x-682) node[right]{\LARGE $L_7$};
			\draw[color=blue, domain=-5:5] plot (-2.5, \x) node[right]{\LARGE $L_7$};
			\draw[color=blue, domain=3.5:-3.25] plot (\x, -0.779914* \x+2.33974) node[left]{\LARGE $L_8$};
			\draw[color=purple, domain=6:-3] plot (\x,  0.774597*\x-2.32379) node[left]{\LARGE $L_9$};
%			\draw[color=purple, domain=6.5:-3] plot (\x,  0.774597*\x-2.32379) node[left]{\LARGE $L_9$};
%			\draw[color=purple, domain=3.5:-3] plot (\x,  0.774597*\x-2.32379) node[left]{\LARGE $L_9$};

%			\draw[color=red, domain=2:-2.75] plot (\x,  0.774597 *\x+1.44721) node[left]{\LARGE $ADO$};
			\draw[color=red, domain=4.75:-2.75] plot (\x,  0.774597 *\x+1.44721) node[left]{\LARGE $ADO$};

			\fill[color=blue] (3,0) circle (3pt);
			\fill[color=blue] (-2.5,4.33) circle (3pt);
			\fill[color=blue] (-2.5,-4.33) circle (3pt);

			\fill[color=green] (1.71781, 1) circle (3pt);
			\fill[color=green] (-2.44569, 1) circle (3pt);
			\fill[color=green] (-0.363943, -2.6057) circle (3pt);

			\fill[color=red] (-.57735,1) circle (3pt);
			\fill[color=red] (-1.86834, 0) circle (3pt);

			\fill[color=red] (4.75, 5.13) circle (3pt);
			
			\end{tikzpicture}}% a=3, c=(3/4)(1+\sqrt{5})
		\hspace{.5cm}
		\scalebox{.5}{
			\begin{tikzpicture}[domain=-5:5,scale=1]

%			\draw (0,0) circle (5cm);
			\draw (0,0) circle (7cm);
			\draw (-7,0) node[left] {\LARGE $L_{10}$};
			\draw[domain=-2.25:7] plot (\x, 0) node[right]{\LARGE $L_1$};
			\draw[domain=.25:-3.5] plot (\x, -1.73* \x) node[above]{\LARGE $L_2$};
			\draw[domain=1.25:-3.5] plot (\x,  1.73*\x) node[below]{\LARGE $L_3$};

			\fill[color=black] (0,0) circle (3pt);
			\fill[color=black] (7,0) circle (3pt);
			\fill[color=black] (-3.5,6.06) circle (3pt);
			\fill[color=black] (-3.5,-6.06) circle (3pt);

			\draw[color=darkgreen, domain=-3:6.25] plot (\x, 1) node[right]{\LARGE $L_4$};
			\draw[color=green, domain=1.25:-2.5] plot (\x, -1.73* \x+1.23607) node[above]{\LARGE $L_5$};
			\draw[color=green, domain=-.5:-5] plot (\x,  1.73*\x+5.17148) node[left]{\LARGE $L_6$};

			\draw[color=blue, domain=1.25:-3] plot (\x, 0.765974*\x+0.895602) node[right]{\LARGE $L_7$};
%			\draw[color=blue, domain=-5:5] plot (-2.5, \x) node[right]{\LARGE $L_7$};
			\draw[color=blue, domain=3.5:-3.25] plot (\x, -0.184897* \x+0.554692) node[left]{\LARGE $L_8$};
			\draw[color=purple, domain=6:-3] plot (\x,  -0.774597*\x+2.32379) node[left]{\LARGE $L_9$};
%			\draw[color=purple, domain=3.5:-3] plot (\x,  -0.774597*\x+2.32379) node[left]{\LARGE $L_9$};

%			\draw[color=red, domain=2:-2.75] plot (\x,  -0.774597 *\x+0.552786) node[left]{\LARGE $ADO$};
			\draw[color=red, domain=5.8:-2.75] plot (\x,  -0.774597 *\x+0.552786) node[left]{\LARGE $ADO$};

			\fill[color=blue] (3,0) circle (3pt);
			\fill[color=blue] (-0.358524, 0.620982) circle (3pt);
			\fill[color=blue] (0.927051, 1.6057) circle (3pt);

			\fill[color=green] (-2.40841, 1) circle (3pt);
			\fill[color=green] (0.136294, 1) circle (3pt);
			\fill[color=green] (-1.13606, 3.20378) circle (3pt);

			\fill[color=red] (-0.57735, 1) circle (3pt);
			\fill[color=red] (0.713644, 0) circle (3pt);

			\fill[color=red] (5.8, -3.95) circle (3pt);

			\end{tikzpicture}}% a=3, c=(3/4)(1-\sqrt{5})
			}
		\caption{Arrangement $(10_3)_7.ADO$ with geometric realizations from the two different irreducible components of its Galois conjugate moduli space. The realization on the left corresponds to the values of $a = 3$ and $c = \frac{3}{4}(1+\sqrt{5})$, and the realization on the right corresponds to the values $a = 3$ and $c = \frac{3}{4}(1-\sqrt{5})$}  %$\mathcal{M}_\mathcal{A}^\C$.}
		\label{fig:ADO}
		\end{center}
	
\end{figure}

The geometric realizations of $(10_3)_7.ADO$ in $\mathbb{CP}^2$ with coordinates $[x:y:z]$, up to a projective transformation, can be described by the equation
\begin{align}
\begin{split}
(y)(\sqrt{3}x+y)(\sqrt{3}x-y)(y-z)(\sqrt{3}(b+c)x+(b+c)y+(2\sqrt{3}bc-c)z) & \\
(\sqrt{3}bx-by+(a+2b+\sqrt{3}ab)z) (\sqrt{3}(b+c)x-(c-b)y+2\sqrt{3}bcz) (-\sqrt{3}bx-(a+b)y+\sqrt{3}abz)  & \\
(\sqrt{3}cx-(a+c)y-\sqrt{3}acz) (z) ( \sqrt{3}(b + c) x - (2 \sqrt{3} bc - b - 3 c)y + (2\sqrt{3}bc - 2 c)z) 
& = 0,
\end{split}
\end{align}
where $a$ is a complex number with a finite number of exceptions, $b=-\tfrac{4c^2 + a c}{a + 2 c - 2 \sqrt{3} c^2}$, and $c^\pm=\tfrac{a}{4}(1 \pm \sqrt{5})$. This shows that we have two irreducible components in the moduli space: one corresponding to $\tfrac{a}{4}(1 + \sqrt{5})$ and another corresponding to $\tfrac{a}{4}(1 - \sqrt{5})$. Since we have one free parameter, each of the irreducible components is one-dimensional.

\end{example}

\begin{example}
	\label{ex:AEIKO}
	The arrangement $\mathcal{A} = (10_3)_1.AEIKO$ has arrangement table given in Table \ref{tab:AEIKO}. Its moduli space can be parameterized by $a$ and $b$ satisfying $a^\pm=(\frac{3\pm\sqrt{5}}{2})b$, where $b$ is a complex number with a finite number of exceptions.%$(a + \frac{-3+ \sqrt{5}}{2}b)(a + \frac{-3- \sqrt{5}}{2}b) = 0$.
	We also know that $\Aut((10_3)_1.AEIKO) \cong F_{20} \cong \langle (12345),(1243) \rangle$, the Frobenius group of order 20. 
	
	\begin{table}[!htp]
		%		\centering
		\begin{tabular}{ccccccccccc}
			$L_1$ & $L_2$ & $L_3$ & $L_4$ & $L_5$ & $L_6$ & $L_7$ & $L_8$ & $L_9$ & $L_{10}$ & $L_{11}$\\ \hline
			1		&1	&1	&8	&2	&3	&2	&3	&4	&5	&$A$\\
			2		&4	&6	&9	&4	&5	&6	&7	&6	&7	&$E$\\
			3		&5	&7	&0	&8	&8	&9	&9	&0	&0	&$I$\\
			$A$	&$E$&$I$&$A$&$K$&$I$&$E$&$O$&$O$&$K$&$K$\\
			& 	& 	& 	& 	& 	& 	& 	& 	& 	&$O$\\
		\end{tabular}
		\caption{The arrangement table for $(10_3)_1.AEIKO$,  whose moduli space is two Galois conjugate points}%disconnected $\mathcal{M}_\mathcal{A}^\C$.}
		\label{tab:AEIKO}
\end{table}

\end{example}

\medskip

\textbf{Organization. } Section \ref{sect:background} describes in more detail the objects with which we work, including the moduli space of a combinatorial line arrangement.

Section \ref{sect:construction} details the one-line extension construction. 

Section \ref{sect:extensionsof103} applies the one-line extension construction to all ten $(10_3)$ configurations. We calculate that, up to isomorphism, there are three hundred twenty-one one-line extensions of $(10_3)$ configurations through three double points, one hundred fifty-one one-line extensions of $(10_3)$ configurations through four double points, and twenty-three one-line extensions of $(10_3)$ configurations through five double points.

Section \ref{sect:moduli} discusses how the moduli space is calculated and the algebraic techniques to determine its irreducibility. We conclude that there are seventy-six one-line extensions of $(10_3)$ configurations with a reducible moduli space up to complex conjugation. 

Section \ref{sec:corrections} offers minor corrections 
to the work by the first author with Amram, Teicher, and Ye \cite{amram}.

\medskip

\textbf{Acknowledgements. } The authors would like to thank the Undergraduate Research Summer Institute at Vassar College, an in-house research experience for undergraduates, for their funding of a portion of this research while the second author was a rising senior at Vassar during the summer of 2017.  %Thanks go to John McCleary and 2018 URSI students Jordan Buhmann, Alexander May, and Shiyu Shu.

\begin{table}[h!]
	\scalebox{.9}{
\begin{tabular}{|lccc|}
\hline
Arrangement	$\mathcal{A}$			& $\lvert \mathcal{M}_\mathcal{A} \rvert$ & $\lvert \mathcal{M}_\mathcal{A}^\mathbb{C}\rvert$& $\Aut(\mathcal{A})$	\\
\hline
1.AEIK				&	$\infty^1$	&$\infty^1$	& $\Z/4\Z$ \\
\hline
2.AENO  $\cong$ 7.ADLO& 2 &2  &  \\
\hline
3.BDHL	$\cong$ 9.BDIM	& 3 &2  & \\
3.BDIK	$\cong$ 6.AENO	& 3 &2  & \\
3.BDIL						& 3 &2  & \\
3.BDKL	$\cong$ 9.BDMN	& 3 &3  & \\
3.BFJM						& 3 &2  & \\
3.DHLO						& 3  &2 & \\
\hline
5.AEIK						& 5 &3  & \\
5.AEJO						& 3 &2  & \\
5.AEKN						& 5 &3  & \\
5.AENO						& 4 &3  & \\
5.AFIK						& 4 &2  & \\
5.AFIO						& 4 &3  & $\Z/2\Z$ \\
5.AFKL						& 4 &3  & \\
5.AFLO						& 4 &4  & \\
5.BDHL	$\cong$ 5.BFIK	& 4 &2  & \\
5.BDHN  $\cong$ 7.AIJL	& 3&2	& \\
5.BDIK	$\cong$ 5.BDKL	$\cong$ 10.AEIJ	& 6 &3  & \\
5.BDJL	$\cong$ 5.BEIK	& 3 &2  & \\
5.BDKN	$\cong$ 10.ADKM	& 5 &3  & \\
5.BEGJ						& 5 &3  & \\		
5.BEGK						& 5 &4  & \\		
5.BEGN						& 4 &2  & \\		
5.BEKN						& 4&2  & \\
5.BFGK						& 4&3   & \\
5.BFKL						& 3 &2  & \\
5.BGJL						& 3 &2  & \\
5.BGKL						& 5 &3  & \\
5.BGKN						& 4 &3  & \\
5.DHLO						& 3 &3  & \\
5.DHMN  $\cong$ 7.AILO	&3&3	& \\
5.DHNO  $\cong$ 7.AFIL	& 2 & 2	& \\
5.DIKM						& 4 & 2   & \\
5.DKMN						& 2  & 2 & $\Z/2\Z$ \\
 & & & \\
 & & & \\
 & & & \\
	\footnotesize{\hspace{1em}$^1$ Two 1-dimensional components} & & & \\
\hline
\end{tabular}
\hspace{5mm}
\begin{tabular}{|lccc|}
\hline
Arrangement	$\mathcal{A}$			& $\lvert \mathcal{M}_\mathcal{A} \rvert$ & $\lvert \mathcal{M}_\mathcal{A}^\mathbb{C}\rvert$& $\Aut(\mathcal{A})$	\\
\hline
6.AEHO						& 4  &2 & $\Z/2\Z$ \\
6.AFIJ						& 3 &2  & \\
6.AFIO						& 4 &2  & $\Z/2\Z$ \\
\hline
7.ADIL						& 3&2   & \\
7.ADIM	$\cong$ 7.AFIM	& 2  &2 & \\
7.AEIJ						& 3&2   & \\
7.AEIM						& 2&2  & $\Z/2\Z$ \\
7.AEIO						& 2&2   &  \\
7.AEJM						& 2&2   & $\Z/2\Z$ \\
7.AEJN						& 2&2   & \\
7.AIJM						& 4&2   & \\
7.BIJM						& 3&2   & \\
\hline
8.ADIM						& 4&3   & \\
8.AEGM						& 4&2   & \\
8.AEIM						& 7&4   & \\
8.AFIJ	$\cong$ 9.ADIM	& 5  &3 & \\
8.AFIL						& 5&3   & \\
8.AFJO						& 5&3   & \\
8.AIJM						& 5&3   & \\
8.AILM						& 5&4   & \\
8.BDIM						& 4 &2  & \\
8.BFIJ						& 5&3   & \\
8.BFIK	$\cong$ 9.ADLN	& 5 &4  & \\
8.BFJO						& 4&2   & \\
8.BFKO						& 3&2   & \\
8.BIJM						& 4&2   & \\
8.BIKM						& 6&4   & \\
8.CFIJ	$\cong$ 9.AEIM	& 4 &3  & $\Z/2\Z$ \\
\hline
9.ADIL						& 5 & 3   & \\
9.ADMN						& 5 &3  & \\
9.AEGM						& 3&2   & \\	
9.BDHM						& 3&2   &  \\
9.BDHO						& 3&2   &  \\
9.BGMN						& 3&2   & \\
\hline
10.ADHM						& 4&3   & \\
10.ADHN						& 4&3   & $\Z/2\Z$ \\
10.ADIN						& 5&4   & \\
10.AEGM						& 6&4  & \\
10.AEGN						& 2&2   & $\Z/2\Z$ \\
\hline
\end{tabular}
}
%\begin{flushleft}
%	\footnotesize{\hspace{5em}$^1$ Two 1-dimensional components}
%\end{flushleft}
\caption{ The list of 74 arrangements with reducible moduli space modulo complex conjugation obtained from one-line extensions of $(10_3)$ arrangements appearing in Theorem \ref{thm:four}, whose automorphism groups are trivial unless otherwise noted}
\label{tab:74cases}
\end{table}

\section{Background}\label{sect:background}
	
We can describe a combinatorial line arrangement as a collection of ``points" and ``lines" along with incidence relations between them. The following definition has been adapted from the definition of \textit{combinatorial configuration} in Gr\"unbaum's textbook \emph{Configurations of Points and Lines} \cite{grunbaum}. 
	
	\begin{definition}
		A \textbf{combinatorial line arrangement} $\mathcal{A} = (\mathcal{P}, \mathcal{L})$ consists of a finite set of \textbf{points} $\mathcal{P}$ and a finite set of \textbf{lines} $\mathcal{L}$, which are subsets of $\mathcal{P}$. We also require that the intersection of two lines is at most one point. Our convention in constructing non-reductive arrangements also requires that each point in $\mathcal{P}$ appears in at least three elements of $\mathcal{L}$. 
If a point $P\in \mathcal{P}$ is on the line $L\in \mathcal{L}$, we say that $P$ is \textbf{incident to} $L$ or that $L$ is \textbf{incident to} $P$.

We refer to the intersection of exactly two lines as a \textbf{double point}, and we define the \textbf{set of double points of $\mathcal{A}$} to be
	\[
		\Doubles(\mathcal{A}) = \left\{\{L, L'\} \in \binom{\mathcal{L}}{2} \mid L \cap L'  = \emptyset \right\}.
	\]
The reason why these 2-tuples are called double points is that we will be attempting to realize these arrangements in projective space, where all pairs of lines intersect exactly once. Using our convention that each point in $\mathcal{P}$ appears in at least three elements of $\mathcal{L}$, we see that if $L \cap L' = \emptyset$, then the intersection of $L$ and $L'$, in the realization in projective space, is not a point that appears in at least three elements of $\mathcal{L}$, so this intersection shall be named a double point since $L$ and $L'$ are the only two lines passing through this intersection.

On the same note, We define \textbf{points of higher multiplicity} as the elements of $\mathcal{P}$. A \textbf{triple point} is an element in $\mathcal{P}$ that appears in exactly three elements of $\mathcal{L}$.

		These arrangements can be presented in an \textbf{arrangement table}, in which the headers are the names of the lines and the columns contain the names of the points incident to each line.

	\end{definition}

	\begin{example}
\label{ex:fano}
		The well-known Fano arrangement is a combinatorial line arrangement whose arrangement table can be found in Table \ref{tab:fano}.

        \begin{table}[h!]
%		\begin{center}
			\begin{tabular}{c c c c c c c}
				$L_1$ & $L_2$ & $L_3$ & $L_4$ & $L_5$ & $L_6$ & $L_7$ \\ \hline
				$P_1$ & $P_1$ & $P_1$ & $P_2$ & $P_2$ & $P_3$ & $P_3$ \\
				$P_2$ & $P_4$ & $P_6$ & $P_4$ & $P_5$ & $P_4$ & $P_5$ \\
				$P_3$ & $P_5$ & $P_7$ & $P_6$ & $P_7$ & $P_7$ & $P_6$
			\end{tabular}
%		\label{tab:fano}
    \caption{An arrangement table for the Fano arrangement}
	\label{tab:fano}
	%		\end{center}
        \end{table}
		
		We see that the arrangement table presentation of the Fano arrangement corresponds to the usual geometric presentation given in Figure \ref{fig:fano}.
		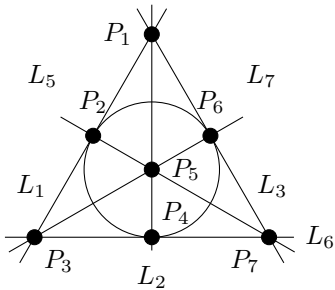
\begin{figure}[!hbtp]
			\begin{center}
				\begin{tikzpicture}
				[every node/.style={draw,circle,inner sep=2pt,fill=black}, scale=0.6]
				%			\foreach \a in {1,2,3} {
				%				\node at (\a*360/3-30: 3cm) (\a) {};
				%			}
				\node [label = right:$P_5$] at (0,0) {};
				
				\node [label = left:$P_1$] at (360/3-30: 3cm) (1) {};
				\node [label = below right:$P_3$] at (2*360/3-30: 3cm) (2) {};
				\node [label = below left:$P_7$] at (3*360/3-30: 3cm) (3) {};

				\draw [name path=l1, shorten <=-.5cm, shorten >=-.5cm] (1)--node [white, very near end, above left = .04cm, text=black] {$L_1$} coordinate[pos = .5](m1)(2);
				\draw [name path=l2, shorten <=-.5cm, shorten >=-.5cm] (1)--node [white, very near end, above right = .04cm, text=black] {$L_3$} coordinate[pos = .5]coordinate[pos = .5](m2)(3);
				\draw [name path=l3, shorten <=-.5cm, shorten >=-.5cm] (2)--node [white, very near end, right = .8cm, text=black] {$L_6$}coordinate[pos = .5](m3)(3);
				\draw [name path=l4, shorten <=-.5cm, shorten >=-.5cm] (m3)--node [white, very near start, below = .5cm, text=black] {$L_2$}(1);
				\draw [name path=l5, shorten <=-.5cm, shorten >=-.5cm] (m2)--node [white, very near start, above right = 1cm, text=black] {$L_7$}(2);
				\draw [name path=l6, shorten <=-.5cm, shorten >=-.5cm] (m1)--node [white, very near start, above left = 1cm, text=black] {$L_5$}(3);
				
				\node [label = above:$P_2$] at (m1) {};
				\node [label = above:$P_6$] at (m2) {};
				\node [label = above right:$P_4$] at (m3) {};
				
				\draw (0,0) circle (1.5cm);
				\end{tikzpicture}
			\end{center}
			\caption{The Fano arrangement with the middle circle as the combinatorial ``line'' $L_4$}
			\label{fig:fano}
		\end{figure}
	\end{example}

\subsection{Moduli space of a combinatorial line arrangement}
\label{subsec:MA}

 With this in mind, we are able to manipulate these combinatorial line arrangements by adding or removing a line.

\begin{definition}
	Let $\mathcal{A} = (\mathcal{P}, \mathcal{L})$ be a combinatorial line arrangement with $\mathcal{L} = \{L_1, L_2, \dots, L_n\}$, and let $L \subseteq \Doubles(\mathcal{A})$ be a subset of the set of double points. We define the \textbf{one-line extension of $\mathcal{A}$ by the line $L$}, written $\mathcal{A}\cup L$, to be 
the line arrangement $(\mathcal{P} \cup L, \mathcal{L}')$, where $\mathcal{L}' = \{L_1', L_2', \dots, L_n', L\}$ and $L_i' = L_i \cup D$ if there exists $D \in L$ such that $L_i \in D$ and $L_i' = L_i$ otherwise. 
\end{definition}
Intuitively, this is adding the line $L$ through its specified points and attaching double points in $\mathcal{A}$ that have now turned into triple points to the appropriate lines.

A similar construction involves removing a line from an arrangement.
\begin{definition}
	Let $\mathcal{A} = (\mathcal{P}, \mathcal{L})$ be a combinatorial line arrangement and $L \in \mathcal{L}$. We define \textbf{$\mathcal{A}$ minus $L$}, written $\mathcal{A} \setminus L$, to be 
	\[
		\mathcal{A} \setminus L = (\mathcal{P}, \mathcal{L} \setminus \{L\})
	\]
	with the convention that double points are omitted.
\end{definition}
	Intuitively, this is removing the line $L$ from $\mathcal{A}$ while removing triple points that are now double points from $\mathcal{P}$ and the corresponding lines.

The following notion of an isomorphism for combinatorial line arrangements is used to identify those we deem to have the same combinatorial information.

\begin{definition}
	Let $\mathcal{A} = (\mathcal{P}, \mathcal{L})$ and $\mathcal{A}' = (\mathcal{P}', \mathcal{L}')$ be two combinatorial line arrangements. We say that $\mathcal{A}$ and $\mathcal{A}'$ are \textbf{isomorphic}, denoted by $\mathcal{A}\cong\mathcal{A}'$, if there exists function $\varphi:\mathcal{P} \cup \mathcal{L} \to \mathcal{P}' \cup \mathcal{L}'$ such that $\varphi|_\mathcal{P}: \mathcal{P} \to \mathcal{P}'$ and $\varphi|_\mathcal{L}: \mathcal{L} \to \mathcal{L}'$ are both bijections. Additionally, for all $L \in \mathcal{L}$ and $P, Q \in L$, we have that $\varphi(P), \varphi(Q) \in \varphi(L)$. Also, if $L, L' \in \mathcal{L}$ are such that $L \cap L' = \{P\}$ for some point $P \in \mathcal{P}$, then $\varphi(L) \cap \varphi(L') = \{ \varphi(P) \}$. Such a function $\varphi$ is called an \textbf{isomorphism}. Also denote by $\Aut(\mathcal{A}) = \{\varphi:\mathcal{A}\to \mathcal{A} \mid \text{$\varphi$ is an isomorphism}  \}$ the \textbf{automorphism group} of $\mathcal{A}$.
\end{definition}

\begin{example}
\label{ex:fanoAut}
It is known 
that the automorphism group for the Fano arrangement is $\GL(3, \mathbb{F}_2)$. 
\end{example}

	The main tool we use to analyze combinatorial line arrangements is the moduli space. The moduli space $\mathcal{M}_\mathcal{A}$ is the space of all geometric realizations of the combinatorial line arrangement $\mathcal{A}$ in $\mathbb{CP}^2$. In order to define the moduli space, we need to define a geometric realization for a combinatorial line arrangement.
	
	\begin{definition}
		Let $\mathcal{A}$ be a combinatorial line arrangement with lines $L_1, L_2, \ldots, L_n$. A \textbf{geometric realization of $\bf \mathcal{A}$ in} $ \mathbb{CP}^2$ is a collection of lines $\ell_1, \ell_2, \ldots, \ell_n$ in $\mathbb{CP}^2$ such that for any subset of $S \subseteq \{1,2, \dots, n\}$ with $\lvert S \rvert \geq 3$, we have $\bigcap_{i \in S} \ell_i$ is nonempty if and only if $\bigcap_{i \in S} L_i$ is nonempty. We say that $\mathcal{A}$ is \textbf{geometrically realizable in} $\mathbb{CP}^2$ if there exists a geometric realization of $\mathcal{A}$ in $\mathbb{CP}^2$.
	\end{definition}

Consider the complex projective line $\ell$ with equation $ax+by+cz=0$.  Then the complex projective point \textbf{dual} to this line in $\mathbb{CP}^2$ is $\ell^*=[a:b:c] \in (\mathbb{CP}^2)^*$.  This is a convenient way to represent lines as points for the sake of notation.

	Now we are ready to introduce the moduli space.

	\begin{definition}
		The \textbf{ordered moduli space of a combinatorial line arrangement $\mathcal{A}$} is defined to be 
		\[
		\mathcal{M}_\mathcal{A} = \{(\ell_1^*, \ell_2^*, \ldots, \ell_n^*) \in ((\mathbb{CP}^2)^*)^n \mid \text{$(\ell_1, \ell_2, \ldots, \ell_n)$ is a geometic realization of $\mathcal{A}$ in $\mathbb{CP}^2$} \}/\PGL(3,\mathbb{C}),
		\]
		which are all geometric realizations of $\mathcal{A}$ in $\mathbb{CP}^2$, up to a projective transformation.  
	\end{definition}

We refer to this throughout as simply the moduli space of an arrangement, noting that there exists a related notion of the unordered moduli space, obtain by a quotient by the automorphism group. The moduli space is endowed the topology induced by the Zariski topology on $((\mathbb{CP}^2)^*)^n$. We will also sometimes endow the moduli space with the finer Euclidean topology.  An important quotient of the moduli space is $\mathcal{M}_\mathcal{A}^\mathbb{C}$, which is defined to be the quotient of $\mathcal{M}_\mathcal{A}$ under complex conjugation.

\begin{example}
\label{ex:fanoModuli}
The Fano arrangement is not geometrically realizable in $\mathbb{CP}^2$; its moduli space is empty.  
These notions are equivalent in general.  
\end{example}

Given a combinatorial line arrangement $\mathcal{A}$, it is a necessary condition for its moduli space $\mathcal{M}_\mathcal{A}$ to be reducible for two geometric realizations of $\mathcal{A}$ to be different topologically. If $\mathcal{M}_\mathcal{A}$ is irreducible, then we can apply Randell's Isotopy Theorem to show that all of the geometric realizations of $\mathcal{A}$ are the same topologically.  

\begin{theorem}
(Randell's Isotopy Theorem \cite{randell})\label{irr}  
Two combinatorially isomorphic arrangements $\mathcal{R}_1$ and $\mathcal{R}_2$ connected by a 1-parameter family of isomorphic arrangements have complements in $\mathbb{CP}^2$ that are diffeomorphic. Furthermore, $(\mathbb{CP}^2, \mathcal{R}_1)$ and $(\mathbb{CP}^2, \mathcal{R}_2)$ are of the same topological type.
\end{theorem}

We also consider the moduli space modulo complex conjugation due the following result from Cohen and Suciu.

\begin{theorem} (Cohen-Suciu \cite[Theorem 3.9]{CS}) \label{cohen-suciu}
	The braid monodromies of complex conjugated curves are equivalent.
\end{theorem}

	 For the purpose of classifying irreducibility of moduli spaces, it turns out that we only need to analyze the moduli spaces of \textit{non-reductive} combinatorial line arrangements. This is a consequence of the following result, which shows that we can remove a line incident to at most two points of higher multiplicity and use the irreducibility of the moduli space of the smaller arrangement to deduce the irreducibility of the moduli space of the original, larger arrangement.
	
	\begin{theorem} (Nazir-Yoshinaga \cite[reframed Lemma 3.2]{ny})
		Let $\mathcal{A} = (\mathcal{P}, \{L_1, L_2, \ldots, L_n\})$ be a combinatorial line arrangement with $L_n$ being incident to at most two points of higher multiplicity, and let $\mathcal{A}' = \mathcal{A} \setminus L_n$. Then $\mathcal{M}_\mathcal{A}$ is irreducible if $\mathcal{M}_{\mathcal{A}'}$ is irreducible. 
	\end{theorem}

We use the one-line extension construction, which we describe in the next section, to provide examples of combinatorial line arrangements with a reducible moduli space.  These come from known arrangements which we discuss in the next subsection.

\subsection{$(n_3)$ configurations}
\label{subsec:n3}
In anticipation of our later construction, we introduce these examples from the literature.

\begin{definition} (see for example \cite{grunbaum})
	An $(n_k)$ \textbf{configuration} is a combinatorial line arrangement with $n$ lines and $n$ points where each line is incident to $k$ points and each point is incident to $k$ lines. 
\end{definition}

We are specifically interested in $(n_3)$ configurations. Table \ref{n3} provides a non-exhaustive enumeration of $(n_3)$ configurations as found in Gr\"unbaum's textbook \cite[Theorem 2.2.1]{grunbaum}, in which he cites others \cite{martinetti, daub, gropp}. It is worth noting that since $(n_3)$ configurations can be seen as $3$-regular, $3$-uniform hypergraphs of girth $\geq 3$, any finite automorphism group is achievable with an $(n_3)$ configuration \cite{hyper}.

\begin{table}[htb]
	\begin{center}
		\begin{tabular}{|l||c|c|c|c|c|c|c|}
			\hline
			$n$ & $\leq 6$ & 7 & 8 & 9 & 10 & 11 & 12 \\
			\hline
			number of $(n_3)$ configurations & 0 & 1 & 1 & 3  & 10  & 31  & 229 \\
			\hline
		\end{tabular}
	\end{center}
	\caption{Number of $(n_3)$ configurations %up to isomorphism 
	as found in \cite[Theorem 2.2.1]{grunbaum}}\label{n3}
\end{table}

\iffalse
		\begin{tabular}{rr}
			$n$ & number of $(n_3)$ configurations \\\hline
			$\leq 6$ & 0\\
			7 & 1 \\
			8 & 1 \\
			9 & 3 \\ 
			10 & 10 \\
			11 & 31 \\
			12 & 229
		\end{tabular}
\fi

The construction in \cite{amram} begins with one of the three $(9_3)$ configurations and considers the set of all double points.  Then all subsets of cardinality three are considered, and this set is quotiented out by the automorphism group of the $(9_3)$ configuration.  This gives the possible combinatorial line arrangements of 10 lines obtained by adding a line passing through three double points.

Table \ref{93+} at the end of Section \ref{sec:corrections} gives the eleven arrangements that result from this construction. Following the naming convention in \cite{amram}, the numbers correspond to points in the original $(9_3)$ configuration, and the capital letters correspond to double points in the original $(9_3)$ configuration that have become triple points after the tenth line has been added.

\begin{remark}
\label{rem:corrections}
We note here that in \cite{amram} there are fourteen arrangements given, but in Section \ref{sec:corrections} we show that three of these are redundant due to arrangement $(9_3)_1$ having a larger automorphism group than acknowledged in that work.
\end{remark}

Table \ref{ten3} lists all ten $(10_3)$ arrangements.  For ease of naming new combinatorial line arrangements, we name each of the 15 double points  with the upper case letters $A$ through $O$ for each of the $(10_3)$ configurations $\mathcal{A}$.  These are given explicitly in Table \ref{tab:doubles}.

\begin{table}[htb]
	\setlength{\tabcolsep}{5pt}
	\begin{center}
		\begin{tabular}{|c|c|}\hline
			\makecell{
				$\text{Aut}((10_3)_1) \cong S_5$.\\
				\begin{tabular}{cccccccccc}
					$L_1$ & $L_2$ & $L_3$ & $L_4$ & $L_5$ & $L_6$ & $L_7$ & $L_8$ & $L_9$ & $L_{10}$ \\ \hline
					1&1&1&8&2&3&2&3&4&5\\
					2&4&6&9&4&5&6&7&6&7\\
					3&5&7&0&8&8&9&9&0&0
			\end{tabular}}
			&
			\makecell{$\text{Aut}((10_3)_2) \cong D_{12}$.\\
				\begin{tabular}{cccccccccc}
					$L_1$ & $L_2$ & $L_3$ & $L_4$ & $L_5$ & $L_6$ & $L_7$ & $L_8$ & $L_9$ & $L_{10}$ \\ \hline
					1&1&1&8&2&3&2&3&4&5\\
					2&4&6&9&4&7&6&5&6&7\\
					3&5&7&0&8&8&9&9&0&0
			\end{tabular}}
			\\\hline
			
			\makecell{
				$\text{Aut}((10_3)_3) \cong \Z/2\Z \times \Z/2\Z$.\\
				\begin{tabular}{cccccccccc}
					$L_1$ & $L_2$ & $L_3$ & $L_4$ & $L_5$ & $L_6$ & $L_7$ & $L_8$ & $L_9$ & $L_{10}$ \\ \hline
					1&1&1&8&2&3&2&3&4&5\\
					2&4&6&9&4&6&7&5&6&7\\
					3&5&7&0&8&8&9&9&0&0
			\end{tabular}}
			
			&
			\makecell{
				$\text{Aut}((10_3)_4) \cong S_4$.\\
				\begin{tabular}{cccccccccc}
					$L_1$ & $L_2$ & $L_3$ & $L_4$ & $L_5$ & $L_6$ & $L_7$ & $L_8$ & $L_9$ & $L_{10}$ \\ \hline
					1&1&1&8&2&3&2&3&4&5\\
					2&4&6&9&4&6&5&7&6&7\\
					3&5&7&0&8&8&9&9&0&0
			\end{tabular}}
			\\\hline
			
			\makecell{
				$\text{Aut}((10_3)_5) \cong \Z/2\Z$.\\
				\begin{tabular}{cccccccccc}
					$L_1$ & $L_2$ & $L_3$ & $L_4$ & $L_5$ & $L_6$ & $L_7$ & $L_8$ & $L_9$ & $L_{10}$ \\ \hline
					1&1&1&8&2&3&2&4&3&5\\
					2&4&6&9&4&7&5&6&6&7\\
					3&5&7&0&8&8&9&9&0&0
			\end{tabular}}
			
			&
			\makecell{
				$\text{Aut}((10_3)_6) \cong S_3$.\\
				\begin{tabular}{cccccccccc}
					$L_1$ & $L_2$ & $L_3$ & $L_4$ & $L_5$ & $L_6$ & $L_7$ & $L_8$ & $L_9$ & $L_{10}$ \\ \hline
					1&1&1&8&2&3&2&5&3&4\\
					2&4&6&9&4&7&6&7&5&6\\
					3&5&7&0&8&8&9&9&0&0
			\end{tabular}}
			\\\hline
			
			\makecell{
				$\text{Aut}((10_3)_7) \cong \Z/3\Z$.\\
				\begin{tabular}{cccccccccc}
					$L_1$ & $L_2$ & $L_3$ & $L_4$ & $L_5$ & $L_6$ & $L_7$ & $L_8$ & $L_9$ & $L_{10}$ \\ \hline
					1&1&1&2&4&6&5&3&7&2\\
					2&4&6&8&8&9&7&5&3&4\\
					3&5&7&9&0&0&8&9&0&6
			\end{tabular}}
			
			&
			\makecell{
				$\text{Aut}((10_3)_8) \cong \Z/3\Z$.\\
				\begin{tabular}{cccccccccc}
					$L_1$ & $L_2$ & $L_3$ & $L_4$ & $L_5$ & $L_6$ & $L_7$ & $L_8$ & $L_9$ & $L_{10}$ \\ \hline
					1&1&1&3&5&7&2&6&4&2\\
					2&4&6&8&8&9&7&5&3&4\\
					3&5&7&9&0&0&8&9&0&6
			\end{tabular}}
			\\\hline
			
			\makecell{
				$\text{Aut}((10_3)_9) \cong \Z/4\Z$.\\
				\begin{tabular}{cccccccccc}
					$L_1$ & $L_2$ & $L_3$ & $L_4$ & $L_5$ & $L_6$ & $L_7$ & $L_8$ & $L_9$ & $L_{10}$ \\ \hline
					1&1&1&2&4&6&5&3&2&3\\
					2&4&6&8&8&9&7&5&7&4\\
					3&5&7&9&0&0&8&9&0&6
			\end{tabular}}
			
			&
			\makecell{
				$\text{Aut}((10_3)_{10}) \cong \Z/10\Z$.\\
				\begin{tabular}{cccccccccc}
					$L_1$ & $L_2$ & $L_3$ & $L_4$ & $L_5$ & $L_6$ & $L_7$ & $L_8$ & $L_9$ & $L_{10}$ \\ \hline
					1&1&1&3&2&7&5&6&4&2\\
					2&4&6&8&8&9&7&5&3&4\\
					3&5&7&9&0&0&8&9&0&6
			\end{tabular}}\\\hline
		\end{tabular}
	\end{center}
	\caption{The ten $(10_3)$ configurations with their arrangement tables and automorphism groups}
	\label{ten3}
\end{table}

\begin{table}[h!]\scalebox{.9}{
\begin{tabular}{|c|c|c|c|c|c|c|c|c|c|c|}
\hline
Double	& $(10_3)_1$ & $(10_3)_2$	& $(10_3)_3$ & $(10_3)_4$& $(10_3)_5$ & $(10_3)_6$ & $(10_3)_7$ & $(10_3)_8$ & $(10_3)_9$ & $(10_3)_{10}$ \\
\hline
A & $L_1 \cap L_4$ & $L_1 \cap L_4$ & $L_1 \cap L_4$ & $L_1 \cap L_4$ & $L_1 \cap L_4$ & $L_1 \cap L_4$ & $L_1 \cap L_5$ & $L_1 \cap L_5$ & $L_1 \cap L_5$ & $L_1 \cap L_6$ \\
B & $L_1 \cap L_9$ & $L_1 \cap L_9$ & $L_1 \cap L_9$ & $L_1 \cap L_9$ & $L_1 \cap L_8$ & $L_1 \cap L_8$ & $L_1 \cap L_6$ & $L_1 \cap L_6$ & $L_1 \cap L_6$ & $L_1 \cap L_7$	\\
C & $L_1 \cap L_{10}$ &$L_1 \cap L_{10}$ &$L_1 \cap L_{10}$ &$L_1 \cap L_{10}$ &$L_1 \cap L_{10}$ &$L_1 \cap L_{10}$ & 	$L_1 \cap L_7$  & $L_1 \cap L_8$ & 	$L_1 \cap L_7$  & $L_1 \cap L_8$ \\
D & $L_2 \cap L_4$& $L_2 \cap L_4$ & $L_2 \cap L_4$ & $L_2 \cap L_4$ & $L_2 \cap L_4$&$L_2 \cap L_4$& $L_2 \cap L_4$ & $L_2 \cap L_4$ & $L_2 \cap L_4$ & $L_2 \cap L_4$ \\
E & $L_2 \cap L_7$ & $L_2 \cap L_6$ & $L_2 \cap L_6$ & $L_2 \cap L_6$& $L_2 \cap L_6$ & $L_2 \cap L_6$ & $L_2 \cap L_6$ & $L_2 \cap L_6$ & $L_2 \cap L_6$ & $L_2 \cap L_5$ \\
\hline
F & $L_2 \cap L_8$ & $L_2 \cap L_7$ & $L_2 \cap L_7$& $L_2 \cap L_8$ & $L_2 \cap L_9$ & $L_2 \cap L_7$ & $L_2 \cap L_9$ & $L_2 \cap L_7$ & $L_2 \cap L_9$ & $L_2 \cap L_6$  \\
G & $L_3 \cap L_4$ & $L_3 \cap L_4$& $L_3 \cap L_4$& $L_3 \cap L_4$& $L_3 \cap L_4$& $L_3 \cap L_4$& $L_3 \cap L_4$& $L_3 \cap L_4$& $L_3 \cap L_4$& $L_3 \cap L_4$ \\
H & $L_3 \cap L_5$ & $L_3 \cap L_5$& $L_3 \cap L_5$& $L_3 \cap L_5$& $L_3 \cap L_5$& $L_3 \cap L_5$& $L_3 \cap L_5$& $L_3 \cap L_5$& $L_3 \cap L_5$& $L_3 \cap L_5$ \\
I & $L_3 \cap L_6$ & $L_3 \cap L_8$& $L_3 \cap L_8$ & $L_3 \cap L_7$ & $L_3 \cap L_7$ & $L_3 \cap L_9$ & $L_3 \cap L_8$ & $L_3 \cap L_9$ & $L_3 \cap L_8$ & $L_3 \cap L_9$ \\
J & $L_5 \cap L_8$ & $L_5 \cap L_8$& $L_5 \cap L_8$ & $L_5 \cap L_8$& $L_5 \cap L_9$ & $L_5 \cap L_8$ & $L_4 \cap L_9$ & $L_4 \cap L_{10}$ & $L_4 \cap L_{10}$& $L_4 \cap L_{10}$ \\
\hline
K & $L_5 \cap L_{10}$& $L_5 \cap L_{10}$& $L_5 \cap L_{10}$& $L_5 \cap L_{10}$& $L_5 \cap L_{10}$  & $L_5 \cap L_9$	& $L_5 \cap L_8$ & $L_5 \cap L_{10}$ &	$L_5 \cap L_8$	&	$L_5 \cap L_8$	  \\
L & $L_6 \cap L_7$& $L_6 \cap L_7$& $L_6 \cap L_7$& $L_6 \cap L_7$& $L_6 \cap L_7$& $L_6 \cap L_7$& $L_6 \cap L_7$ & $L_6 \cap L{10}$& $L_6 \cap L_7$& $L_6 \cap L_{10}$	\\
M & $L_6 \cap L_9$ & $L_6 \cap L_9$& $L_6 \cap L_{10}$ & $L_6 \cap L_{10}$ & $L_6 \cap L_8$ &$L_6 \cap L_{10}$& $L_7 \cap L_{10}$ & $L_7 \cap L_8$ & $L_7 \cap L_{10}$ & $L_7 \cap L_9$ \\
N & $L_7 \cap L_{10}$ & $L_7 \cap L_{10}$ & $L_7 \cap L_9$& $L_7 \cap L_9$ &$L_7 \cap L_9$ & $L_7 \cap L_9$ & $L_8 \cap L_{10}$ & $L_7 \cap L_9$ & $L_8 \cap L_9$ & $L_7 \cap L_{10}$ \\
O & $L_8 \cap L_9$ & $L_8 \cap L_9$ & $L_8 \cap L_9$ &$L_8 \cap L_9$ & $L_8 \cap L_{10}$ & $L_8 \cap L_{10}$ & $L_9 \cap L_{10}$ & $L_8 \cap L_9$ & $L_9 \cap L_{10}$ & $L_8 \cap L_9$ \\
\hline
\end{tabular}}
\caption{ Labels of the double points in the $(10_3)$ arrangements }
\label{tab:doubles}
\end{table}

We highlight two geometric pictures of arrangements that better display the symmetry of their respective automorphism group:  the arrangements $(10_3)_7$ and $(10_3)_{10}$ in Figure \ref{fig:1037}.% and the arrangement $(10_3)_{10}$ in Figure \ref{fig:10310}.

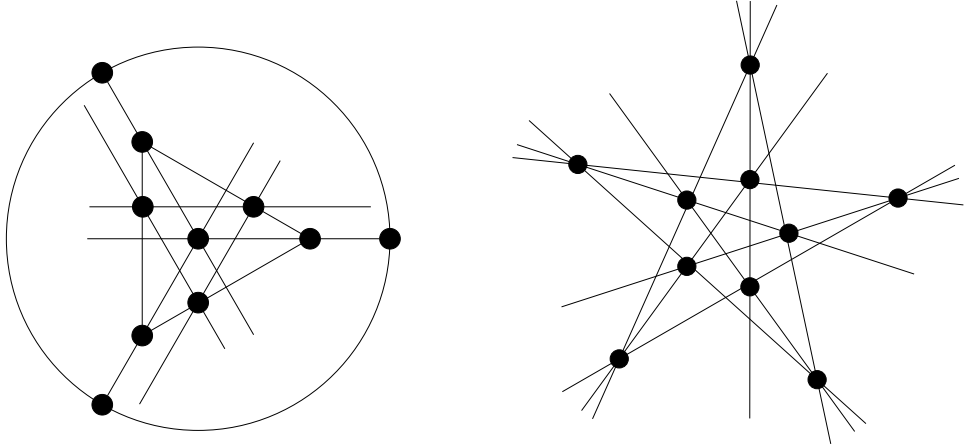
\begin{figure}[!htb]
	\begin{center}
		\scalebox{.85}{
			\begin{tikzpicture}
				[every node/.style={draw,circle,inner sep=4pt,fill=black,color=black,text opacity=1,scale=.55}]
				\foreach \a in {1,2,3} {
					\node at (\a*360/3-120: 3cm) (\a) {};
				}
				\foreach \a in {4,5,6} {
					\node at (\a*360/3-120: 1.75cm) (\a) {};
				}
				\foreach \a in {7,8,9} {
					\node at (\a*360/3+30-120: 1cm) (\a) {};
				}
				\node at (0: 0cm) (10) {};
				\draw (0,0) circle (3cm);
				
				\node [color = white, text = black, above left=1mm, fill opacity=0, draw opacity = 0] at (1) {2};
				\node [color = white, text = black, below=1mm, fill opacity=0, draw opacity = 0] at (2) {4};	
				\node [color = white, text = black, right=1mm, fill opacity=0, draw opacity = 0] at (3) {6};
				\node [color = white, text = black, below=1mm, fill opacity=0, draw opacity = 0] at (4) {3};
				\node [color = white, text = black, right=1mm, fill opacity=0, draw opacity = 0] at (5) {5};	
				\node [color = white, text = black, left=1mm, fill opacity=0, draw opacity = 0] at (6) {7};
				\node [color = white, text = black, above=1mm, fill opacity=0, draw opacity = 0] at (7) {9};	
				\node [color = white, text = black, below left=1mm, fill opacity=0, draw opacity = 0] at (8) {8};
				\node [color = white, text = black, left=1mm, fill opacity=0, draw opacity = 0] at (9) {0};
				\node [color = white, text = black, below=1mm, fill opacity=0, draw opacity = 0] at (10) {1};	
				
				\draw [shorten >= -0cm, shorten <= -1.9cm] (10)--(1);
				\draw [shorten >= -0cm, shorten <= -1.9cm] (10)--(2);
				\draw [shorten >= -0cm, shorten <= -1.9cm] (10)--(3);
				\draw [shorten >= -1cm, shorten <= -2cm] (7)--(8);
				\draw [shorten >= -1cm, shorten <= -2cm] (8)--(9);
				\draw [shorten >= -1cm, shorten <= -2cm] (9)--(7);
				\draw [shorten >= -.3cm, shorten <= -.3cm] (4)--(5);
				\draw [shorten >= -.3cm, shorten <= -.3cm] (5)--(6);
				\draw [shorten >= -.3cm, shorten <= -.3cm] (6)--(4);
				
				\node [color = white, text = black, fill opacity=0, draw opacity = 0] at (355: 2.5cm) {$L_1$};	
				\node [color = white, text = black, fill opacity=0, draw opacity = 0] at (115: 2.5cm) {$L_2$};	
				\node [color = white, text = black, fill opacity=0, draw opacity = 0] at (235: 2.5cm) {$L_3$};	
				\node [color = white, text = black, fill opacity=0, draw opacity = 0] at (200: 1.2cm) {$L_7$};	
				\node [color = white, text = black, fill opacity=0, draw opacity = 0] at (90: 1.2cm) {$L_8$};	
				\node [color = white, text = black, fill opacity=0, draw opacity = 0] at (330: 1.2cm) {$L_9$};	
				\node [color = white, text = black, fill opacity=0, draw opacity = 0] at (165: 2cm) {$L_4$};	
				\node [color = white, text = black, fill opacity=0, draw opacity = 0] at (45: 2cm) {$L_6$};	
				\node [color = white, text = black, fill opacity=0, draw opacity = 0] at (285: 2cm) {$L_5$};	
				\node [color = white, text = black, fill opacity=0, draw opacity = 0] at (180: 3.2cm) {$L_{10}$};
				
			\end{tikzpicture}
		}
		\hspace{.5cm}
		\scalebox{.75}{
			
			\begin{tikzpicture}
				[every node/.style={draw,circle,inner sep=4pt,fill=black,color=black,text opacity=1,scale=.55}, node distance=.9cm]
				\foreach \a in {1,2,3,4,5} {
					\node at (\a*360/5: 1cm) (\a) {};
				}
				\foreach \a in {6,7,8,9,10} {
					\node at (\a*360/5+12: 3cm) (\a) {};
				}
				%		\draw [color=white] (0,0) circle (5cm);

				\draw[color=white] (-1,-3.5) -- (1,-3.5);
				
				\node [color = white, text = black, above left = 1mm, fill opacity=0, draw opacity = 0] at (1) {0};
				\node [color = white, text = black, below left = 1mm, fill opacity=0, draw opacity = 0] at (2) {5};	
				\node [color = white, text = black, below = 1mm, fill opacity=0, draw opacity = 0] at (3) {2};
				\node [color = white, text = black, right = 1mm, fill opacity=0, draw opacity = 0] at (4) {9};
				\node [color = white, text = black, above right = 1mm, fill opacity=0, draw opacity = 0] at (5) {1};

				\node [color = white, text = black, left = 1mm, fill opacity=0, draw opacity = 0] at (6) {7};
				\node [color = white, text = black, above right = 1mm, fill opacity=0, draw opacity = 0] at (7) {4};
				\node [color = white, text = black, above left = 1mm, fill opacity=0, draw opacity = 0] at (8) {8};	
				\node [color = white, text = black, above right = 1mm, fill opacity=0, draw opacity = 0] at (9) {6};
				\node [color = white, text = black, above left = 1mm, fill opacity=0, draw opacity = 0] at (10) {3};

				\draw [shorten >= -1.3cm, shorten <= -2.5cm] (4)--(6);
				\draw [shorten >= -1.3cm, shorten <= -2.5cm] (3)--(10);
				\draw [shorten >= -1.3cm, shorten <= -2.5cm] (2)--(9);
				\draw [shorten >= -1.3cm, shorten <= -2.5cm] (1)--(8);
				\draw [shorten >= -1.3cm, shorten <= -2.5cm] (5)--(7);
				
				\draw [shorten >= -1.33cm, shorten <= -1.33cm] (6)--(8);
				\draw [shorten >= -1.33cm, shorten <= -1.33cm, name path=m] (7)--(9);
				\draw [shorten >= -1.33cm, shorten <= -1.33cm, name path=l] (8)--(10);
				\draw [shorten >= -1.33cm, shorten <= -1.33cm] (9)--(6);
				\draw [shorten >= -1.33cm, shorten <= -1.33cm] (10)--(7);
				
				\node [color = white, text = black, fill opacity=0, draw opacity = 0] at (15: 4.5cm) {$L_1$};	
				\node [color = white, text = black, fill opacity=0, draw opacity = 0] at (-15: 4.5cm) {$L_2$};	
				\node [color = white, text = black, fill opacity=0, draw opacity = 0] at (95: 4.5cm) {$L_3$};	
				\node [color = white, text = black, fill opacity=0, draw opacity = 0] at (22: 4.5cm) {$L_4$};	
				\node [color = white, text = black, fill opacity=0, draw opacity = 0] at (55: 4.5cm) {$L_5$};	
				\node [color = white, text = black, fill opacity=0, draw opacity = 0] at (85: 4.5cm) {$L_6$};	
				\node [color = white, text = black, fill opacity=0, draw opacity = 0] at (75: 4.5cm) {$L_7$};	
				\node [color = white, text = black, fill opacity=0, draw opacity = 0] at (130: 4.5cm) {$L_8$};	
				\node [color = white, text = black, fill opacity=0, draw opacity = 0] at (165: 4.5cm) {$L_9$};	
				\node [color = white, text = black, fill opacity=0, draw opacity = 0] at (150: 4.5cm) {$L_{10}$};

			\end{tikzpicture}
		}

		\caption{A geometric picture of $(10_3)_{7}$ on the left showing its $\mathbb{Z}/3\Z$ automorphism group, and a geometric picture of $(10_3)_{10}$ on the right showing the $\Z/5\Z$ subgroup of its $\mathbb{Z}/10\Z$ automorphism group}\label{fig:1037}
	\end{center}
\end{figure}

\section{The One-Line Extension Construction}
\label{sect:construction}

	We first describe the one-line extension construction in full generality. Then we describe the one-line extension construction through a specified number of double points. Lastly, we discuss the effects of this construction on the moduli spaces. 
	
	\begin{definition}
		Let $\mathcal{A} = (\mathcal{P}, \mathcal{L})$ be a combinatorial line arrangement. We define the \textbf{set of one-line extensions of $\mathcal{A}$} as 
		\[
			\OL(\mathcal{A}) = \left\{ \mathcal{A} \cup L \,\middle\vert\, L \in \binom{\mathcal{P} \cup \Doubles(\mathcal{A})}{k}, k \in \Z, k > 0, \mathcal{A} \cup L \text{ is a line arrangement} \right\}/\sim,
		\]
	where $A \cup L \sim A \cup L'$ if the two arrangements are combinatorially isomorphic.
	\end{definition}
	In other words, $\OL(\mathcal{A})$ is the set of all combinatorial line arrangements with $\left|\mathcal{L} \right| + 1$ lines containing $\mathcal{A}$ as a subarrangement, up to isomorphism. The requirement that $\mathcal{A} \cup L$ is a line arrangement ensures that every pair of lines intersect no more than once. 
	
	An interesting family of one-line extensions are extensions purely through double points. Some, but not all, isomorphisms between different extensions are induced by the automorphisms of $\mathcal{A}$. 
	\begin{definition}
	Let $\mathcal{A}$ be a combinatorial line arrangement. We define the \textbf{ set of extension lines of $\mathcal{A}$ by $k$ double points} as
	\[
	\OLExt(k, \mathcal{A}) = \left\{L \in \binom{\Doubles(\mathcal{A})}{k} \,\middle\vert\, \text{$\mathcal{A} \cup L$ is a line arrangement} \right\}/\Aut(\mathcal{A}).
	\]
\end{definition}

We take $k\geq3$ to avoid producing reductive arrangements.

Again, to make sure that $\mathcal{A} \cup L$ is a line arrangement, we have to make sure that every pair of lines intersect no more than once.

	\begin{definition}
The following set contains all the possible combinatorial line arrangements constructed by adding a line through $k$ double points in $\mathcal{A}$ up to isomorphism:
	\[
	\OLExtArrs(k, \mathcal{A}) = \{\mathcal{A} \cup L \mid [L] \in \OLExt(k, \mathcal{A}) \}/\sim,
	\]
	where $A \cup L \sim A \cup L'$ if the two arrangements are combinatorially isomorphic. We call this set \textbf{one line extensions of $\mathcal{A}$ by $k$ double points}. Each element is well-defined since for each $[L] \in \OLExt(k, \mathcal{A})$, if $[L] = [L']$, the automorphism of $\mathcal{A}$ induces an automorphism between $A \cup L$ and $A \cup L'$.
\end{definition}

\begin{remark}
	Note that it is possible that $[L] \neq [L']$ as elements of $\OLExt(k, \mathcal{A})$ yet $\mathcal{A} \cup L \cong \mathcal{A} \cup L'$. 
\end{remark}

\begin{example} Table \ref{tab:Self-exchangeExample} gives an example of a pair of isomorphic arrangements constructed using different classes of $\OLExt(3, \mathcal{A})$, where $\mathcal{A}$ is a $(10_3)$ configuration.
	
	\begin{table}[!htbp]
		{\setlength{\tabcolsep}{.4em} 	\begin{tabular}{cc}

				$(10_3)_5.BDL$ & 			$(10_3)_5.BIK$ \\
				
				\begin{tabular}{ccccccccccc}
					$L_1$ & $L_2$ & $L_3$ & $L_4$ & $L_5$ & $L_6$ & $L_7$ & $L_8$ & $L_9$ & $L_{10}$ & $L_{11}$ \\ \hline
					1&1&1&8&2&3&2&4&3&5&B\\
					2&4&6&9&4&7&5&6&6&7&D\\
					3&5&7&0&8&8&9&9&0&0&L\\
					B&D& &D& &L&L&B
				\end{tabular}
				&
				
				\begin{tabular}{ccccccccccc}
					$L_1$ & $L_2$ & $L_3$ & $L_4$ & $L_5$ & $L_6$ & $L_7$ & $L_8$ & $L_9$ & $L_{10}$& $L_{11}$ \\ \hline
					1&1&1&8&2&3&2&4&3&5&B\\
					2&4&6&9&4&7&5&6&6&7&I\\
					3&5&7&0&8&8&9&9&0&0&K\\
					B& &I& &K& &I&B& &K
				\end{tabular}
				
		\end{tabular}}
		\caption{Two isomorphic arrangements in $\OLExtArrs(3, (10_3)_5)$}
		\label{tab:Self-exchangeExample}
	\end{table}
	
	An isomorphism $\varphi:(10_3)_5.BDL \to (10_3)_5.BIK$ given in Table \ref{tab:Self-exchangeExampleIsomorphism}. Note that $\varphi(\{L_1, L_2, \dots, L_{10}\}) \neq \{L_1, L_2, \dots, L_{10}\}$ or, equivalently, $\varphi(L_{11}) \neq L_{11}$, so $(10_3)_5.BDL$ and $(10_3)_5.BIK$ would not have been identified via the quotient of $\OLExt(3, \mathcal{A})$ by $\Aut((10_3)_5)$. However, the two arrangements are identified since they are combinatorially isomorphic.

	\begin{table}[h!]
		%		\begin{center}
			\begin{tabular}{|c|c|c|c|c|c|c|c|c|c|c|c|c|c|c|}
				\hline
				Point $p$		&0&1&2&3&4&5&6&7&8&9&B&D&L\\\hline
				$\varphi(p)$	&8&6&I&7&B&9&3&0&K&2&1&4&5\\\hline
				\hline
				Line $L$	&	$L_1$ & $L_2$ & $L_3$ & $L_4$ & $L_5$ & $L_6$ & $L_7$ & $L_8$ & $L_9$ & $L_{10}$ & $L_{11}$ & & \\\hline
				$\varphi(L)$	&	$L_3$ & $L_8$ & $L_9$ & $L_5$ & $L_{11}$ & $L_{10}$ & $L_7$ & $L_1$ & $L_6$ & $L_{4}$ & $L_{2}$ & &\\\hline
			\end{tabular}
			%\end{center}
			\caption{ The explicit isomorphism between $(10_3)_5.BDL$ and $(10_3)_5.BIK$}
			\label{tab:Self-exchangeExampleIsomorphism}
		\end{table}
		
	\end{example}

	To detect this phenomenon, we rely on the following lemma:
	
	\begin{lemma}\label{self-exchange}
		Fix $k \in \Z$ such that $k \geq 3$. Let $\mathcal{A} = (\mathcal{P}, \mathcal{L})$ be a combinatorial line arrangement, and let $[L], [L'] \in \OLExt(k, \mathcal{A})$ such that $[L] \neq [L']$. If there exists an isomorphism $\varphi: \mathcal{A} \cup L \to \mathcal{A} \cup L'$, then there exists a line $\ell \in \mathcal{L}$ such that $(\mathcal{A} \cup L') \setminus \ell \cong \mathcal{A}$. 
	\end{lemma}
	
	\begin{proof}
		Since $[L] \neq [L']$, we know that $\varphi(L) \neq L'$, meaning $\varphi(L) = \ell$ for some $\ell \in \mathcal{L}$, so $\mathcal{A} = (\mathcal{A} \cup L) \setminus L \cong (\mathcal{A} \cup L') \setminus \ell$. %Thus $\mathcal{A} \cong (\mathcal{A} \cup L' ) \setminus \ell$.
	\end{proof}

When considering one-line extensions of different combinatorial arrangements, it is also possible to have one-line extensions of non-isomorphic arrangements be isomorphic. In other words, it is possible for two non-isomorphic combinatorial line arrangements $\mathcal{A}_1 \not\cong \mathcal{A}_2$ to have $\mathcal{A}_1' \in \OLExtArrs(k, \mathcal{A}_1)$ and $\mathcal{A}_2' \in \OLExtArrs(k', \mathcal{C}_2)$ such that $\mathcal{A}_1' \cong \mathcal{A}_2'$.

\begin{example}
	The following is an example of a pair of isomorphic arrangements: one comes from some $\OLExtArrs(3, \mathcal{A})$ and the other comes from some $\OLExtArrs(3, \mathcal{B})$, where $\mathcal{A}$ and $\mathcal{B}$ are distinct, non-isomorphic $(10_3)$ configurations:
	\begin{table}[!htbp]\label{tab:General-exchangeExample}
		{\setlength{\tabcolsep}{.4em} 	\begin{tabular}{cc}

				$(10_3)_1.AEM$ & 			$(10_3)_6.KLO$ \\
				
				\begin{tabular}{ccccccccccc}
					$L_1$ & $L_2$ & $L_3$ & $L_4$ & $L_5$ & $L_6$ & $L_7$ & $L_8$ & $L_9$ & $L_{10}$ & $L_{11}$ \\ \hline
					1&1&1&8&2&3&2&3&4&5&A\\
					2&4&6&9&4&5&6&7&6&7&E\\
					3&5&7&0&8&8&9&9&0&0&M\\
					A&E& &A& &M&E& &M&
				\end{tabular}
				&
				
				\begin{tabular}{ccccccccccc}
					$L_1$ & $L_2$ & $L_3$ & $L_4$ & $L_5$ & $L_6$ & $L_7$ & $L_8$ & $L_9$ & $L_{10}$& $L_{11}$ \\ \hline
					1&1&1&8&2&3&2&5&3&4&K\\
					2&4&6&9&4&7&6&7&5&6&L\\
					3&5&7&0&8&8&9&9&0&0&O\\
					& & & &K&L&L&O&K&L
				\end{tabular}
				
		\end{tabular}}
		\caption{Two isomorphic arrangements: one comes from $\OLExtArrs(3, (10_3)_1)$ and the other comes from $\OLExtArrs(3, (10_3)_6)$} 
	\end{table}
	
	An isomorphism $\varphi:(10_3)_1.AEM \to (10_3)_6.KLO$ is given in Table \ref{tab:General-exchangeExampleIsomorphism}.
	
	\begin{table}[h!]
		%		\begin{center}
			\begin{tabular}{|c|c|c|c|c|c|c|c|c|c|c|c|c|c|c|}
				\hline
				Point $p$		&0&1&2&3&4&5&6&7&8&9&A&E&M\\\hline
				$\varphi(p)$	&2&5&O&7&K&3&4&1&L&6&9&0&8\\\hline
				\hline
				Line $L$	&	$L_1$ & $L_2$ & $L_3$ & $L_4$ & $L_5$ & $L_6$ & $L_7$ & $L_8$ & $L_9$ & $L_{10}$ & $L_{11}$ & & \\\hline
				$\varphi(L)$	&	$L_8$ & $L_9$ & $L_2$ & $L_7$ & $L_{11}$ & $L_{6}$ & $L_7$ & $L_3$ & $L_5$ & $L_{1}$ & $L_{4}$ & &\\\hline
			\end{tabular}
			%\end{center}
			\caption{ The explicit isomorphism between $(10_3)_1.AEM $ and $(10_3)_6.KLO$}
			\label{tab:General-exchangeExampleIsomorphism}
		\end{table}
	\end{example}

	To detect when one-line extensions from different arrangements are isomorphic, we have the following tool. 
	
	\begin{lemma}\label{general exchange}
		Let $\mathcal{A}_1 = (\mathcal{P}_1, \mathcal{L}_1)$ and $\mathcal{A}_2 = (\mathcal{P}_2, \mathcal{L}_2)$ be two arrangements such that $\mathcal{A}_1 \not\cong \mathcal{A}_2$, and let $[L] \in \OLExt(k, \mathcal{A}_1)$ and $[L'] \in \OLExt(k, \mathcal{A}_2)$. If there exists an isomorphism $\varphi: \mathcal{A}_1 \cup L \to \mathcal{A}_2 \cup L'$, then there exists a line $\ell \in \mathcal{L}_2$ such that $\mathcal{A}_2 \cup L' \setminus \ell \cong \mathcal{A}_1$. 
	\end{lemma}

	\begin{proof}
		 If $\varphi(L) = L'$, then it must be the case that $\varphi(\mathcal{A}_1) = \mathcal{A}_2$, which goes against our hypothesis that $\mathcal{A}_1 \not\cong \mathcal{A}_2$. We may then assume that $\varphi(L) \neq L'$ and $\varphi(L) = \ell$ for some $\ell \in \mathcal{L}_2$. Then $\mathcal{A}_1 = (\mathcal{A}_1 \cup L) \setminus L \cong (\mathcal{A}_2 \cup L') \setminus \ell$ by $\varphi$. 
	\end{proof}

Empirically, the one-line extension construction is effective since the construction generally reduces the dimension of the moduli space and a zero dimensional space is reducible as long as it is not empty or a singleton. The dimension of the moduli space is reduced after a one-line extension since requiring three points to be collinear is an additional constraint on the moduli space unless those three points are already collinear. Requiring four points to be collinear gives two additional constraints in the generic case, and requiring fives points to be collinear gives three additional constraints in the generic case. 

\section{One-Line Extensions of $(10_3)$ Configurations}\label{sect:extensionsof103}

Out of the eleven %(see Section \ref{sec:corrections} on why this differs from \cite{amram})
possible one-line extensions of $(9_3)$ configurations by three double points, six have a reducible moduli space. With this as motivation, we apply one-line extensions to all $(10_3)$ configurations in hopes of producing combinatorial line arrangements of 11 lines with a reducible moduli space.

We now apply the one-line extension construction through $k$ double points to all of the $(10_3)$ configurations, for $k = 3,4,5$. Note that $k \geq 6$ is impossible since 6 double points require 12 other lines. The construction is detailed in Algorithm \ref{const}.

\begin{algorithm}[!htb]
	\caption{Enumeration Algorithm}\label{const}
	\begin{algorithmic}
		\REQUIRE A list $\mathfrak{C}$ of all $(10_3)$ configurations up to isomorphism and their automorphism groups, and a value for $k$
		\ENSURE A list $\mathfrak{A}$ of combinatorial line arrangements of $11$ lines that can be constructed by adding a line through $k$ double points in a configuration in $\mathfrak{C}$, up to isomorphism
		
		\medskip
		\STATE Initialize $\mathfrak{A} := \emptyset$ 
		\STATE Initialize $\mathfrak{F} := \emptyset$, the arrangements needing additional testing
		\FOR {$\mathcal{C} = (\mathcal{P}, \mathcal{L}) \in \mathfrak{C}$}
		\STATE Calculate $\OLExt(k, \mathcal{C})$
		\FOR {$L \in \OLExt(k, \mathcal{C})$}
		\STATE Set $\mathfrak{A} := \mathfrak{A} \cup \{\mathcal{C} \cup L\}$
		\FOR {$L^* \in \mathcal{L}$}
		\IF {\text{$\mathcal{C} \cup L \setminus L^*$ is a $(10_3)$ configuration}}
		\STATE $\mathfrak{F} := \mathfrak{F} \cup \{\mathcal{C} \cup L\}$
		\STATE \textbf{break}
		\ENDIF
		\ENDFOR
		\ENDFOR
		\ENDFOR
		
		\medskip
		\STATE Check for isomorphisms among the arrangements in $\mathfrak{F}$, throwing away isomorphic copies in $\mathfrak{A}$
		\RETURN $\mathfrak{A}$
	\end{algorithmic}
\end{algorithm}

In order to start implementing this construction, we need the automorphism groups of the $(10_3)$ configurations. The generators for the automorphism groups were also provided in \cite{martinetti}. A correspondence between the naming of the configurations in \cite{martinetti} and \cite{grunbaum} is provided in \cite{grunbaum}.  We indicate the automorphism groups in Table \ref{ten3}.

Now we enumerate the one-line extensions of $(10_3)$ configurations through 3, 4, or 5 double points up to isomorphism utilizing Lemmas \ref{self-exchange} and \ref{general exchange}. Fix $k = 3,4,5$. We can calculate $\OLExt(k, (10_3)_i)$ for $i = 1, \dots, 10$ directly. Then we form the set
\[
	\bigcup_{i=1}^{10} \{(10_3)_i \cup L \mid [L] \in \OLExt(k, (10_3)_i \},
\]
which has all of the possible one-line extensions of $(10_3)$ configurations through $k$ double points, but some isomorphism classes might be represented more than once. Lemmas \ref{self-exchange} and \ref{general exchange} imply that such an isomorphism class can only come from arrangements of the form $(10_3)_i \cup L$ such that there exists a line $\ell$ of $(10_3)_i$ where $(10_3)_i \cup L \setminus \ell$ is a $(10_3)$ configuration. This creates a much smaller set to check for pairwise isomorphisms. This enumeration process is detailed in Algorithm \ref{const}.

Using the automorphism groups of all the $(10_3)$ configurations from Table \ref{ten3}, we can apply Algorithm $\ref{const}$. For $k = 3$, the first block of code results in three hundred thirty-seven arrangements. Lemma \ref{self-exchange} identifies a pair of isomorphic arrangements both constructed from $(10_3)_5$, leaving us with a subtotal of three hundred thirty-six arrangements. Lemma \ref{general exchange} identifies fifteen pairs of isomorphic line arrangements, and we are able to conclude that the remaining three hundred twenty-one arrangements are pairwise non-isomorphic. See Table \ref{tab:main3} for more details. For $k= 4$ and $k = 5$, the results of the construction are summarized in Tables \ref{tab:main4} and \ref{tab:main5}. The results in those tables about the reducibility of various notions of the moduli space are obtained by using tools from the next section.

\section{Irreducibility of Moduli Spaces}
\label{sect:moduli}

In order to use results from Section \ref{sect:background}, we must be able to calculate the moduli space of a combinatorial line arrangement. We follow Algorithm 2 in \cite{bokowski} to achieve this. 

Given a combinatorial line arrangement $\mathcal{A}$, we fix a projective basis of four points (or lines), accounting for the quotient by $\PGL(3,\mathbb{C})$. Then we add in a line or point from $\mathcal{A}$ one at a time, parameterizing when needed. We are then left with parameters $v_1, v_2, \dots, v_r$ and polynomials $f_1, f_2, \dots, f_s \in \C[v_1, v_2, \dots, v_r]$ such that $f_i(v_1, v_2, \dots, v_r) = 0$ for all $i = 1,2, \dots, s$ to ensure the collinearity relations. There are also $g_1, g_2, \dots, g_t \in \C[v_1, v_2, \dots, v_r]$ such that $g_j(a_1, a_2, \dots, a_r) \neq 0$ for all $j = 1,2, \dots, t$, since $g_i(v_1, v_2, \dots, v_r) = 0$ would correspond to a degenerate realizations of the arrangement.

\begin{example}\label{ex:ANOFirstParameterization}
		The construction of $(10_3)_5$ in \cite{amram} required the introduction of 3 parameters named $a,b,c$ and the constraint on the parameters is \[-a^2 b^2 c + a^2 b^2 + a^2 b c - a^2 c + 2 a b c^2 - 3 a b c - a c^2 + 2 a c - b c^2 + b c + c^2 - c = 0.\]
	
	One of the new arrangements is obtained by adding a line through the double points $A, N, O$  and has the arrangement table given in Table \ref{tab:1035}.
	\begin{table}[h!]
		%	\begin{center}
			\begin{tabular}{ccccccccccc}
				$L_1$ & $L_2$ & $L_3$ & $L_4$ & $L_5$ & $L_6$ & $L_7$ & $L_8$ & $L_9$ & $L_{10}$ & $L_{11}$ \\ \hline
				1&1&1&8&2&3&2&4&3&5&A\\
				2&4&6&9&4&7&5&6&6&7&N\\
				3&5&7&0&8&8&9&9&0&0&O\\
				A& & &A& & &N&O&N&O
			\end{tabular}
			\caption{An arrangement table for the arrangement $(10_3)_5.ANO$}
			\label{tab:1035}
			%	\end{center}
	\end{table}

	The additional line requires the additional constraint that
	\begin{align*}
		2a^2b^2c - a^2b^2 - 2a^2bc^2 + a^2c^2 + ab^2c^3 - 2ab^2c + ab^2 - 2abc^3 &\\+ abc^2 + abc + 2ac^3 - 2ac^2 - bc^4 + bc^3 + bc^2 - bc + c^4 - 2c^3 + c^2 &= 0.
	\end{align*}
\end{example}

If the system $\{f_i = 0 \mid i = 1,2, \dots s\} \cup \{g_j \neq 0 \mid j =1,2,\dots, t\}$ has no solutions, then the combinatorial line arrangement is not geometrically realizable. If the system does have a solution, then we analyze the irreducibility of the moduli space. In order to look at the irreducibility of the moduli space, we eliminate the factors of each $f_i$ that contradict the system $\{g_j \neq 0 \mid j =1,2,\dots, t\}$. We end up with polynomials $h_1, h_2, \dots, h_m \in \C[v_1, v_2, \dots, v_r]$ so that the moduli space is the variety $V(h_1, h_2, \dots, h_m) \setminus V(g_1, g_2, \dots, g_t)$ and its closure is $V(h_1, h_2, \dots, h_m)$. Then the irreducible components of the moduli space correspond to the irreducible components of $V(h_1, h_2, \dots, h_m)$.

For all of our arrangements, $V(h_1, h_2, \dots, h_m)$ is a finite set or there exists a parameterization such that $m = 1$. This is achieved by selecting a projective basis that minimizes $m$. If $V(h_1, h_2, \dots, h_m)$ is a finite set, then $\mathcal{M}_\mathcal{A}$ is a finite set and the moduli space $\mathcal{M}_\mathcal{A}$ is reducible if it contains at least two points. If $m = 1$ and $h_1$ is irreducible over $\C$, then $V(h_1)$ and therefore $\mathcal{M}_\mathcal{A}^\mathbb{C}$ are irreducible. We will discuss how to decide whether $h_1$ is irreducible over $\C$ later in this section.

\begin{example}\label{ex:ANO}
	In Example \ref{ex:ANOFirstParameterization}, the moduli space is a variety that is described by two polynomials. It would be easier to determine irreducibility if we can reparamaterize the moduli space so that it is described by a singular polynomial. Indeed, we can reparametrize the moduli space by starting with the following four points and coordinates:
	\begin{equation*}
		9:\; [1:0:0]\hspace{10mm}
		2:\; [0:1:0]\hspace{10mm}
		0:\; [0:0:1]\hspace{10mm}
		12:\; [1:1:1]
	\end{equation*}

	Then the points and lines were constructed the following order: \[9, 2, 0, O, L_{10}, L_8, L_7, L_4, 5, L_1, A, L_{11}, N, L_9, 3, 6, L_2, 1, 4, L_3, L_5, 8, 7, L_6.\]
	
	This leads to a new parametrization with only two variables and the single constraint that
	\[
	a^4b^2 + a^4b - 3a^3b^2 - 3a^3b + a^2b^2 + 2a^2b - 2ab - a + 1 = 0,
	\]
	which is easier to interpret, since we will see that this is an irreducible polynomial over $\C$ by Theorem \ref{critc} and Lemma \ref{lem:zirred} so the moduli space is irreducible. 
\end{example}

Suppose $m = 1$ and $h_1$ is reducible over $\C$. Here we look at the intersections of the irreducible components of $V(h_1)$ to see if they are in $\mathcal{M}_\mathcal{A}^\mathbb{C}$. If an intersection point of the irreducible components of $V(h_1)$ is in $\mathcal{M}_\mathcal{A}^\mathbb{C}$, the irreducible components containing this intersection point are in the same Euclidean-connected component.

As we have discussed earlier, for many of our arrangements, we can translate the problem of determining the irreducibility of the moduli space into a problem of determining the irreducibliity of the polynomial describing the moduli space.

\begin{definition}
	A polynomial $f$ over a field $K$ is \textbf{absolutely irreducible} if $f$ is irreducible over any field extension of $K$, or equivalently, $f$ is irreducible over $\overline{K}$, the algebraic closure of $K$. 
\end{definition}

The problem of determining whether $f\in K[x_1, x_2,\ldots,x_n]$ is absolutely irreducible can be translated into a problem about the Newton polytope of $f$, as we are able to put a monoidal operation on Newton polytopes that mimics the multiplication of polynomials.

\begin{definition}
	Let $K$ be a field and let $f \in K[x_1,x_2,\ldots,x_n]$ be a polynomial. Then the \textbf{Newton polytope} of $f$ is defined to be the convex hull of a set of points in $\R^n$:
	\[
	P_f = \textup{Conv} \left\{(u_1, u_2, \ldots, u_n) \,\middle\vert\, \alpha_{u_1,u_2,\ldots,u_n} \neq 0, f = \sum_{j_1,j_2,\ldots,j_n} \alpha_{j_1,j_2,\ldots,j_n} x_1^{j_1}x_2^{j_2}\cdots x_n^{j_n} \right\}.
	\]
\end{definition}

The Newton polytope is defined this way so that the points in $P_f$ provide a geometric interpretation of the combinations of exponents in each term in $f$. Now we want an operation on Newton polytopes that provides an analogy for polynomial multiplication in this geometric interpretation.

\begin{definition}(see for example Bertone-Ch\`eze-Galligo \cite[Definition 5]{bcg})
	If $A_1$ and $A_2$ are two subsets of $\R^n$, then we define their \textbf{Minkowski sum} as 
	\[
		A_1 + A_2 = \{a_1 + a_2 \mid a_1 \in A_1, a_2 \in A_2 \}.
	\]
\end{definition}

In fact, if we view $(K[x_1,x_2,\ldots,x_n], \cdot)$ as a monoid and $(\textup{\textbf{NewP}}_n, +)$, where $\textup{\textbf{NewP}}_n$ is the set of all convex polytopes with nonnegative integer coordinates in $\R^n$, also as a monoid, then $P: K[x_1,x_2,\ldots,x_n] \to \textup{\textbf{NewP}}_n$ defined by $f \mapsto P_f$ is a monoid homomorphism by the following proposition. 

\begin{proposition}(Ostrowski \cite{ostrowski}, as found in Bertone-Ch\`eze-Galligo \cite[Lemma 6]{bcg})
	Let $f,g \in K[x_1,x_2,\ldots,x_n]$. Then $P_{fg} = P_f + P_g$. 
\end{proposition}

Using the language of Newton polytopes, the following proposition leads to a useful criterion for detecting the absolute irreducibility of polynomials. 

\begin{proposition} (\cite{rupprecht})
	If $f \in K[x_1,x_2,\ldots,x_n]$ is absolutely reducible and $f = f_1f_2 \cdots f_s$, then $P_f = P_{f_1} + P_{f_2} + \cdots P_{f_s} = \underbrace{P_{f_1} + \cdots P_{f_1}}_{\text{$s$ times}}$
\end{proposition}

\begin{proof}
	The factors $f_1, f_2, \ldots, f_s$ are conjugates over $K$ so the corresponding Newton polytopes are the same.
\end{proof}

This leads to the following reformulation of a result by \cite{gao} as found in \cite{bcg}.
\begin{theorem}(Gao \cite{gao}, as found in Bertone-Ch\`eze-Galligo \cite{bcg})\label{critc}
	Let $f(x_1, x_2, \ldots, x_n)$ be an irreducible polynomial in $K[x_1, x_2, \ldots,x_n]$, where $K$ is a field. If the Newton polytope has the following convex hull:
	\[P_f = \{(x_1^{(1)}, x_2^{(1)}, \dots, x_n^{(1)}), (x_1^{(2)}, x_2^{(2)}, \dots, x_n^{(2)}), \ldots, (x_1^{(k)}, x_2^{(k)}, \dots, x_n^{(k)}) \}\] and the coordinates of the points on the convex hull are coprime, i.e.
	\begin{equation}
	\label{eq:Gao}
	\gcd(x_1^{(1)}, x_2^{(1)}, \dots, x_n^{(1)}, x_1^{(2)}, x_2^{(2)}, \dots, x_n^{(2)}, \ldots, x_1^{(k)}, x_2^{(k)}, \dots, x_n^{(k)}) = 1,
	\end{equation}
	then $f$ is absolutely irreducible.
\end{theorem}

It is easy to check algorithmically the $\gcd$ condition of Equation \ref{eq:Gao} in Theorem \ref{critc}. However, we require a little bit more theory to see whether $f$ is indeed irreducible over $K$, in order to satisfy the first hypothesis of Theorem \ref{critc}. For the purpose of this paper, we are only looking at polynomials over $K = \mathbb{Q}$. In fact, the polynomials we are considering have coefficients in $\Z$, and due to Gauss's Lemma, a primitive polynomial $f$ irreducibe over $\Z$ is an irreducible polynomial viewed as an element in $\Q[x_1, x_2, \dots, x_n]$ so we need a criteria for a polynomial to be irreducible over $\Z$, such as the following.

\begin{lemma}\label{lem:zirred}
	Let $f \in \Z[x_1, x_2, \ldots, x_n]$. Now let $g(x_1) = f(x_1, t_2, t_3 \dots, t_n)$ for some particular values $t_2, t_3, \dots, t_n \in \Z$. Furthermore, let $p$ be a prime number and define $h(x_1) \in \Z/p\Z[x_1]$ to be $g(x_1)$ with its coefficients reduced modulo $p$. If both $h$ and $f$ have the same degree in the variable $x_1$ and $h$ is irreducible over $\Z/p\Z$, then $f$ is irreducible over $\Z$. 
\end{lemma}

Since there are only a small number of irreducible univariate polynomials over $\Z/p\Z$ of a certain degree for a small prime $p$, it is easy to check if $h$ is irreducible over $\Z/p\Z$ in the above lemma. 

\begin{example}
	The polynomial in Example \ref{ex:ANO} is $f(a,b) := a^4b^2 + a^4b - 3a^3b^2 - 3a^3b + a^2b^2 + 2a^2b - 2ab - a + 1$, whose Newton polytope has vertices $(0,0), (1,0), (2,2), (4,1), (4,2)$. The greatest common divisor of the coordinates of these vertices is 1, so Theorem \ref{critc} implies that $f(a,b)$ is absolutely irreducible as long as $f(a,b)$ is irreducible over $\Q$. 
	
	We consider $f(-1, b) = 5b^2 +8b +2$. We see that $f(-1,b)$ has no roots modulo 7, so $f(-1,b)$ is irreducible over $\Z$ and therefore $f(a,b)$ is irreducible over $\Z$ by Lemma \ref{lem:zirred}. Then Gauss' Lemma implies that $f(a,b)$ is irreducible over $\Q$. 
\end{example}

Theorem \ref{critc} is easy to check and suffices to show absolute irreducibility for the majority of the polynomials with which we are concerned. The following theorem handles the rest of the polynomials we have by removing some points from the Newton Polytope and then applying Theorem \ref{critc}.

\begin{theorem}{(Bertone-Ch\`eze-Galligo \cite[Proposition 9]{bcg}, Kaltofen \cite{kaltofen})}\label{critcmodular}
	Let $f \in \Z[x_1, x_2, \dots, x_n]$ and let $\overline{f}$ be $f$ with its coefficients reduced modulo $p$ for some prime $p$. If $\deg(f) = \deg(\overline{f})$ and $\overline{f}$ is absolutely irreducible, then $f$ is absolutely irreducible.
\end{theorem}

Now Theorems \ref{critc} and \ref{critcmodular} provide us with criteria to test the irreducibility of the polynomial constraints on the parameters of the moduli spaces. This then allows us to determine the number of connected components the moduli space has. The results are shown in Theorems \ref{thm:main}, \ref{thm:four}, and \ref{thm:five}, showing the number of arrangements in each family, classified by their moduli spaces. We break down the tables by which $(10_3)$ configuration each arrangement is a one-line extension of and also by whether the moduli space is irreducible, empty, $\mathcal{M}_\mathcal{A}$ being reducible and $\mathcal{M}_\mathcal{A}^\mathbb{C}$ being irreducible, or $\mathcal{M}_\mathcal{A}^\mathbb{C}$ being reducible. The subtotal simply adds all of the columns together, whereas the total accounts for identifications up to isomorphism.

\section{Corrections}
\label{sec:corrections}

 We now remedy a small error in previous work by the first author with Amram, Teicher, and Ye that leads to three arrangements being counted twice.

	\begin{proposition}
	\label{prop:correction}
		There are only two distinct line arrangements with ten lines obtained by adding a tenth line through three double points in the configuration $(9_3)_1$. 
	\end{proposition}

This replaces the following now-incorrect result:

	\begin{proposition} (Amram-Cohen-Teicher-Ye \cite[Lemma 8.2]{amram})
		There are five line arrangements with ten lines obtained by adding a tenth line through three double points in the configuration $(9_3)_1$. 
	\end{proposition}

	The argument in \cite{amram} relies on the automorphism group of the configuration $(9_3)_1$ to identify isomorphic arrangements with ten lines constructed from the configuration $(9_3)_1$ as described. It appears that \cite{amram} has depicted the configuration $(9_3)_1$ as having an automorphism group isomorphic to $D_6$, the dihedral group of order twelve. However, the automorphism group of the configuration $(9_3)_1$ has order larger than twelve.
	
	\begin{proposition}(Coxeter \cite{pappus})
				The automorphism group of the configuration $(9_3)_1$ is isomorphic to $PG(2,3)$, a group of order one hundred eight.
	\end{proposition}

	A larger automorphism group would mean that we could potentially identify some of the arrangements with ten lines constructed from $(9_3)_1$ as isomorphic. Indeed, this is the case.

We conclude with a proof of Proposition \ref{prop:correction}.

	\begin{proof}
		In \cite{amram}, five line arrangements with 10 lines are said to be generated as described, but we show explicitly that some of the new line arrangements are isomorphic.
		
		Using the names of the arrangements in \cite{amram}, we first show $(9_3)_1.CDI \cong (9_3)_1.CFH$ , whose arrangement tables are given in Table \ref{tab:1stIsom}.

\begin{table}[h!]
%\begin{center}
			\begin{tabular}{cc}

			$(9_3)_1.CDI$ & 			$(9_3)_1.CFH$ \\
			
			\begin{tabular}{cccccccccc}
				$L_1$ & $L_2$ & $L_3$ & $L_4$ & $L_5$ & $L_6$ & $L_7$ & $L_8$ & $L_9$ & $L_{10}$ \\ \hline
				1&1&1&2&3&5&0&0&0&C\\
				2&4&6&4&6&7&4&2&3&D\\
				3&5&7&8&8&8&6&7&5&I\\
				C&I& &D& &C& &I&D&
			\end{tabular}
&
			
			\begin{tabular}{cccccccccc}
				$L_1$ & $L_2$ & $L_3$ & $L_4$ & $L_5$ & $L_6$ & $L_7$ & $L_8$ & $L_9$ & $L_{10}$ \\ \hline
				1&1&1&2&3&5&0&0&0&C\\
				2&4&6&4&6&7&4&2&3&F\\
				3&5&7&8&8&8&6&7&5&H\\
				C&H&F&F&H&C& & & &
			\end{tabular}
	
			\end{tabular}
%\end{center}
\caption{{ Arrangement tables for two isomorphic arrangements $(9_3)_1.CDI$ and $(9_3)_1.CFH$ }}
\label{tab:1stIsom}
\end{table}

		The arrangements $(9_3)_1.CDI$ and $(9_3)_1.CFH$ are isomorphic due to the following isomorphism $\varphi$.  The isomorphism $\varphi$ also induces an automorphism of the lines.  See Table \ref{tab:1stIsomPointsLines}.
		
\begin{table}[h!]
%		\begin{center}
			\begin{tabular}{|c|c|c|c|c|c|c|c|c|c|c|c|c|}
				\hline
				Point $p$&0&1&2&3&4&5&6&7&8&C&D&I\\\hline
				$\varphi(p)$&4&3&1&2&6&8&0&5&7&C&F&H\\\hline
				\hline
				Line $L$	&	$L_1$ & $L_2$ & $L_3$ & $L_4$ & $L_5$ & $L_6$ & $L_7$ & $L_8$ & $L_9$ & $L_{10}$ & & \\\hline
				$\varphi(L)$	&	$L_1$ & $L_5$ & $L_9$ & $L_3$ & $L_8$ & $L_6$ & $L_7$ & $L_2$ & $L_4$ & $L_{10}$ & & \\\hline
			\end{tabular}
		%\end{center}
\caption{{ An explicit isomorphism $\varphi: (9_3)_1.CDI \to (9_3)_1.CFH$}}
\label{tab:1stIsomPointsLines}
\end{table}

		Now we show that $(9_3)_1.CDG \cong_{\varphi_1} (9_3)_1.CDH \cong_{\varphi_2} (9_3)_1.CFG$ , whose arrangement tables are given in Table \ref{tab:2ndIsom}. 
		
		{\setlength{\tabcolsep}{0.25em} 
\begin{table}[h!]
%\begin{center}
\resizebox{\columnwidth}{!}{%
				\begin{tabular}{ccc}
				$(9_3)_1.CDG$&$(9_3)_1.CDH$&$(9_3)_1.CFG$\\

			\begin{tabular}{cccccccccc}
				$L_1$ & $L_2$ & $L_3$ & $L_4$ & $L_5$ & $L_6$ & $L_7$ & $L_8$ & $L_9$ & $L_{10}$ \\ \hline
				1&1&1&2&3&5&0&0&0&C\\
				2&4&6&4&6&7&4&2&3&D\\
				3&5&7&8&8&8&6&7&5&G\\
				C& & &D&G&C& &G&D&
			\end{tabular}
&
			
			\begin{tabular}{cccccccccc}
				$L_1$ & $L_2$ & $L_3$ & $L_4$ & $L_5$ & $L_6$ & $L_7$ & $L_8$ & $L_9$ & $L_{10}$ \\ \hline
				1&1&1&2&3&5&0&0&0&C\\
				2&4&6&4&6&7&4&2&3&D\\
				3&5&7&8&8&8&6&7&5&H\\
				C&H& &D&H&C& & &D&
			\end{tabular}
&
			\begin{tabular}{cccccccccc}
				$L_1$ & $L_2$ & $L_3$ & $L_4$ & $L_5$ & $L_6$ & $L_7$ & $L_8$ & $L_9$ & $L_{10}$ \\ \hline
				1&1&1&2&3&5&0&0&0&C\\
				2&4&6&4&6&7&4&2&3&F\\
				3&5&7&8&8&8&6&7&5&G\\
				C& &F&F&G&C& &G& &
			\end{tabular}
		\end{tabular}
}
%\end{center}
\caption{{ Arrangement tables for three isomorphic arrangements $(9_3)_1.CDG, (9_3)_1.CDH$ and $(9_3)_1.CFG$ }}
\label{tab:2ndIsom}
\end{table}
	}

		All three arrangements are isomorphic due to the isomorphisms $\varphi_1$ and $\varphi_2$.  The isomorphisms $\varphi_1, \varphi_2$ also induce an automorphisms of the lines. See Table \ref{tab:2ndIsomPointsLines}.			\end{proof}
		
\begin{table}[h!]
%		\begin{center}
			\begin{tabular}{|c|c|c|c|c|c|c|c|c|c|c|c|c|}
				\hline
				Point $P$&0&1&2&3&4&5&6&7&8&C&D&G\\\hline
				$\varphi_1(P)$&4&7&5&8&0&2&6&1&3&C&D&H\\\hline
				$\varphi_2(\varphi_1(P))$&6&5&8&7&4&1&0&3&2&C&F&G\\\hline
							\hline
Line $L$	&					$L_1$ & $L_2$ & $L_3$ & $L_4$ & $L_5$ & $L_6$ & $L_7$ & $L_8$ & $L_9$ & $L_{10}$ & & \\\hline
$\varphi_1(L)$	&				$L_6$ & $L_8$ & $L_3$ & $L_9$ & $L_5$ & $L_1$ & $L_7$ & $L_2$ & $L_4$ & $L_{10}$ & & \\\hline
$\varphi_2(\varphi_1(L))$	&	$L_6$ & $L_2$ & $L_9$ & $L_4$ & $L_8$ & $L_1$ & $L_7$ & $L_5$ & $L_3$ & $L_{10}$ & & \\\hline

			\end{tabular}
%\end{center}
\caption{{ Explicit isomorphisms for $(9_3)_1.CDG \cong_{\varphi_1} (9_3)_1.CDH \cong_{\varphi_2} (9_3)_1.CFG$}}
\label{tab:2ndIsomPointsLines}
\end{table}

\begin{table}[!htbp]
	\begin{center}	\setlength{\tabcolsep}{4.5pt}
		\begin{tabular}{|c||c|}
			\hline
			\makecell{ $(9_3)_1.CFH$ \\ \begin{tabular}{cccccccccc}
					$L_1$ & $L_2$ & $L_3$ & $L_4$ & $L_5$ & $L_6$ & $L_7$ & $L_8$ & $L_9$ & $L_{10}$ \\ \hline
					1&1&1&2&3&5&0&0&0&C\\
					2&4&6&4&6&7&4&2&3&F\\
					3&5&7&8&8&8&6&7&5&H\\
					C&H&F&F&H&C& & & &
				\end{tabular} }%\\ $(9_3).i.CFH$} 
			& \makecell{$(9_3)_3.BDF$\\ \begin{tabular}{cccccccccc}
					$L_1$ & $L_2$ & $L_3$ & $L_4$ & $L_5$ & $L_6$ & $L_7$ & $L_8$ & $L_9$ & $L_{10}$ \\ \hline
					1&1&1&8&8&8&9&9&9&B\\
					2&4&6&2&5&3&2&4&3&D\\
					3&5&7&4&6&7&5&7&6&F\\
					&F&B&B& &D&D& &F&
				\end{tabular} }\\ \hline %\\$(9_3).iii.BDF$}\\ \hline
			\makecell{$(9_3)_1.CDG$ \\	\begin{tabular}{cccccccccc}
					$L_1$ & $L_2$ & $L_3$ & $L_4$ & $L_5$ & $L_6$ & $L_7$ & $L_8$ & $L_9$ & $L_{10}$ \\ \hline
					1&1&1&2&3&5&0&0&0&C\\
					2&4&6&4&6&7&4&2&3&D\\
					3&5&7&8&8&8&6&7&5&G\\
					C& & &D&G&C& &G&D& \\
%							\hline
				\end{tabular} }% \\ $(9_3).i.CDG$ } 
			& \makecell{$(9_3)_3.ACG$ \\ \begin{tabular}{cccccccccc}
					$L_1$ & $L_2$ & $L_3$ & $L_4$ & $L_5$ & $L_6$ & $L_7$ & $L_8$ & $L_9$ & $L_{10}$ \\ \hline
					1&1&1&8&8&8&9&9&9&A\\
					2&4&6&2&5&3&2&4&3&C\\
					3&5&7&4&6&7&5&7&6&G\\
					G& &C&A&G& &C& &A&
				\end{tabular} }\\ \hline %\\ $(9_3).iii.ACG$}\\ \hline		
			
			\makecell{ $(9_3)_2.DFI$ \\				\begin{tabular}{cccccccccc}
					$L_1$ & $L_2$ & $L_3$ & $L_4$ & $L_5$ & $L_6$ & $L_7$ & $L_8$ & $L_9$ & $L_{10}$ \\ \hline
					1&1&1&8&8&8&4&3&2&D\\
					2&4&6&4&2&3&7&5&5&F\\
					3&5&7&6&7&9&9&6&9&I\\
					I&D&F&I&D&F& & & &
				\end{tabular} }%\\$(9_3).ii.DFI$}
			& \makecell{$(9_3)_3.AEG$ \\ \begin{tabular}{cccccccccc}
					$L_1$ & $L_2$ & $L_3$ & $L_4$ & $L_5$ & $L_6$ & $L_7$ & $L_8$ & $L_9$ & $L_{10}$ \\ \hline
					1&1&1&8&8&8&9&9&9&A\\
					2&4&6&2&5&3&2&4&3&E\\
					3&5&7&4&6&7&5&7&6&G\\
					G&E& &A&G&E& & &A&
				\end{tabular} }\\ \hline %\\ $(9_3).iii.AEG$} \\ \hline
			
			\makecell{$(9_3)_2.CFI$ \\	\begin{tabular}{cccccccccc}
					$L_1$ & $L_2$ & $L_3$ & $L_4$ & $L_5$ & $L_6$ & $L_7$ & $L_8$ & $L_9$ & $L_{10}$ \\ \hline
					1&1&1&8&8&8&4&3&2&C\\
					2&4&6&4&2&3&7&5&5&F\\
					3&5&7&6&7&9&9&6&9&I\\
					I& &F&I&C&F& &C& &
				\end{tabular} }%\\ $(9_3).ii.CFI$}
			& \makecell{$(9_3)_3.ADG$ \\ \begin{tabular}{cccccccccc}
					$L_1$ & $L_2$ & $L_3$ & $L_4$ & $L_5$ & $L_6$ & $L_7$ & $L_8$ & $L_9$ & $L_{10}$ \\ \hline
					1&1&1&8&8&8&9&9&9&A\\
					2&4&6&2&5&3&2&4&3&D\\
					3&5&7&4&6&7&5&7&6&G\\
					G& & &A&G&D&D& &A&
				\end{tabular} }\\ \hline %\\ 	$(9_3).iii.ADG$}\\ \hline

			\makecell{ $(9_3)_2.DFA$ \\ 				\begin{tabular}{cccccccccc}
					$L_1$ & $L_2$ & $L_3$ & $L_4$ & $L_5$ & $L_6$ & $L_7$ & $L_8$ & $L_9$ & $L_{10}$ \\ \hline
					1&1&1&8&8&8&4&3&2&D\\
					2&4&6&4&2&3&7&5&5&F\\
					3&5&7&6&7&9&9&6&9&A\\
					A&D&F& &D&F&A& & &
				\end{tabular} }%\\ $(9_3).ii.DFA$}
			& \makecell{ $(9_3)_3.BEG$ \\ \begin{tabular}{cccccccccc}
					$L_1$ & $L_2$ & $L_3$ & $L_4$ & $L_5$ & $L_6$ & $L_7$ & $L_8$ & $L_9$ & $L_{10}$ \\ \hline
					1&1&1&8&8&8&9&9&9&A\\
					2&4&6&2&5&3&2&4&3&D\\
					3&5&7&4&6&7&5&7&6&G\\
					G& & &A&G&D&D& &A&
				\end{tabular} }  \\ \hline% \\$(9_3).iii.BEG$}  \\ \hline
			\makecell{ $(9_3)_2.DFH$ \\
				\begin{tabular}{cccccccccc}
					$L_1$ & $L_2$ & $L_3$ & $L_4$ & $L_5$ & $L_6$ & $L_7$ & $L_8$ & $L_9$ & $L_{10}$ \\ \hline
					1&1&1&8&8&8&4&3&2&D\\
					2&4&6&4&2&3&7&5&5&F\\
					3&5&7&6&7&9&9&6&9&H\\
					&D&F&H&D&F& & &H&
				\end{tabular} }% \\ 		$(9_3).ii.DFH$}
				&\\ \hline
		\end{tabular}
	\end{center}
	\caption{The eleven one-line extensions of $(9_3)$ configurations found in \cite{amram}}\label{93+}
\end{table}

%\bibliographystyle{amsalpha}
%\bibliography{references}

\newcommand{\etalchar}[1]{$^{#1}$}
\providecommand{\bysame}{\leavevmode\hbox to3em{\hrulefill}\thinspace}
\providecommand{\MR}{\relax\ifhmode\unskip\space\fi MR }
\providecommand{\MRhref}[2]{%
  \href{http://www.ams.org/mathscinet-getitem?mr=#1}{#2}
}
\providecommand{\href}[2]{#2}

\end{document}